 \newtheorem{open}{Open Problem}

\documentclass[a4paper]{llncs}
\usepackage{amsmath}
\usepackage{amssymb}
\newtheorem{notation}[theorem]{Notation}
\newcommand{\F}{{\mathbb F}}

\newcommand{\var}{{\rm Var}}
\newcommand{\w}{{\rm w_H}}

\newcommand{\rds}{restriction degree stability\ }
\newcommand{\ds}{{\rm deg\_stab}}
\newcommand{\rank}{{\rm rank}_{r-1}}
\newcommand{\Forbidden}{{ Forbidden}}
\newcommand{\Failed}{{ Failed}}
\newcommand{\success}{{ success}}

\allowdisplaybreaks
\newcommand{\SqBinom}[2]{\genfrac{[}{]}{0pt}{}{#1}{#2}}

\usepackage{algorithm}
\usepackage{algpseudocode}
\begin{document}

\title{
On the algebraic degree stability of Boolean functions when restricted to  affine spaces
\thanks{The research of the first author is partly supported by the Norwegian Research Council and  the two
other authors are supported by EPSRC, UK (EPSRC grant EP/W03378X/1)}}
\author{Claude Carlet\inst{1} \and Serge Feukoua\inst{2} \and Ana
S\u{a}l\u{a}gean \inst{3}}
\institute{LAGA, Department of Mathematics, University of
Paris 8 (and Paris 13 and CNRS),
 Saint--Denis cedex 02, France, and University of Bergen, Norway.\
\email{claude.carlet@gmail.com}
\and
Department of Computer Science, University  of
Loughborough, UK; ENSTP
Yaound\'e, Cameroon.
\email{S.C.Feukoua-Jonzo@lboro.ac.uk; sergefeukoua@gmail.com}
\and
Department of Computer Science, University
of Loughborough, UK.\
\email{A.M.Salagean@lboro.ac.uk}}
\title{The stability of the algebraic degree of Boolean functions when restricted to  affine spaces
\thanks{The research of the first author is partly supported by the Norwegian Research Council and  the two
other authors are supported by EPSRC, UK (EPSRC grant EP/W03378X/1)}}


\maketitle \thispagestyle{empty}
\begin{abstract}
We study the
$n$-variable Boolean functions which keep their algebraic degree
unchanged when they are restricted to any (affine) hyperplane, or more generally
to any affine space of a given co-dimension $k$.
For cryptographic applications it is of interest to determine functions $f$ which have a relatively high degree and also maintain this degree when restricted to affine spaces of co-dimension $k$ for $k$ ranging from 1 to as high a value as possible. This highest value will be called  the \rds of $f$, denoted by $\ds(f)$.
We give several necessary and/or sufficient conditions for $f$ to maintain its degree on spaces of co-dimension $k$; we show that this property is related to the property of having ``fast points'' as well as to other properties and parameters. The value of $\ds(f)$ is determined for functions of degrees $r\in
\{1,2,n-2,n-1,n\}$ and for functions which are direct sums of monomials; we also determine the symmetric functions which maintain their degree on any hyperplane.
Furthermore, we give an explicit formula for the number of functions which maintain their degree on all hyperplanes.  Finally, using our previous results and some computer assistance, we determine the behaviour of all the functions in 8 variables, therefore determining the optimal ones (i.e. with highest value of $\ds(f)$) for each degree.
\end{abstract}
 \noindent
\section{Introduction}
 A Boolean function $f$ is  a function from  the set
$\{0,1\}^n$ to  $\{0,1\}$
(see the recent monograph \cite{Carlet} for more details on Boolean functions).
Boolean functions are used in many research areas; of particular relevance to our work is their use in 
 acryptography and sequences (where Boolean functions are omnipresent in the construction of symmetric ciphers - stream ciphers and block ciphers, see \cite{ruepel,Carlet}) as well as in algebraic coding theory
 (with Reed Muller codes and Kerdock codes).
 However, the existence of some  cryptanalysis techniques (such as correlation, fast correlation and algebraic attacks on stream ciphers, linear and differential attacks on block ciphers, and ``guess and determine'' versions of these attacks, see e.g. \cite{Matsui,Courtois,Carlet}), leads to the use  in symmetric cryptography of Boolean functions satisfying some mandatory criteria. 

These  cryptographic functions should in particular have a high
algebraic degree. Indeed, almost all  cryptosystems using  Boolean
functions (as in the filter model or the combiner model  of stream
ciphers) can be attacked if the considered functions have low
algebraic degree. For instance, they can be vulnerable to fast
algebraic attacks (see e.g. \cite{Courtois1,Carlet}). It is also important
that this algebraic degree remains high even if the function is
restricted to an affine hyperplane or to an affine space of
co-dimension $k$ to avoid ``guess and determine attacks '' where the
attacker would make assumptions resulting in the fact that the input
to the function is restricted to a particular affine space. Note
that the algebraic degree of the restriction of a given function $f$
to any affine space is less than or equal to the algebraic degree
$\deg(f)$ of the global function. 

In \cite{Brier-Langevin}, the authors used invariants based on the restrictions of a function to all linear hyperplanes in order to classify functions of degree 3 in 9 variables.
In~\cite{CS}, the authors describe an infinite
class of functions  whose algebraic degree remains unchanged when
they are restricted to any affine hyperplane. However, general characterizations of functions with this property were not given.

In this paper, we start a systematic study of the functions which keep their degree unchanged when restricted to affine spaces of a certain co-dimension $k$. Of course, we can restrict ourselves to homogeneous functions (that is, instead of working on the Reed-Muller code $RM(r,n)$, work on the quotient ring $RM(r,n)/RM(r-1,n)$).
The case of $k=1$ corresponds to hyperplanes, and it is an important first step towards arbitrary co-dimensions. We start in Section~\ref{sec:defs} by defining a ``degree-drop'' space of a function $f$ as being an affine space $A$ such that the restriction $f_{\mid A}$ has degree strictly lower than the degree of $f$. We denote by $K_{k,r,n}$ the set of functions of degree $r$ in $n$ variables which do not have any degree-drop space of co-dimension $k$. We also define the notion of \rds for a function $f$, denoted by $\ds(f)$, as being the largest co-dimension $k$ for which $f$ has no degree-drop space of co-dimension $k$. The largest value of $\ds(f)$ over all functions of degree $r$ in $n\geq r$ variables will be denoted by $\ds(r,n)$. Functions which reach this value would be optimal from the point of view of their degree stability, and therefore of interest in cryptographic constructions.

After proving some basic properties of these notions, we give in Section~\ref{sec:nec-suff} several necessary and/or sufficient conditions for a function to have no degree-drop space of co-dimension $k$. Notably, we prove that a function $f$ has a degree-drop hyperplane if and only if it is affine equivalent to a function $f_1$ of the form $x_1g +h$ where $\deg(h)<\deg(x_1g)$, i.e.\ there is a variable which appears in all the monomials of maximum degree of $f_1$. A generalization of this condition to arbitrary $k$ is also given.  We then find  more constructive sufficient conditions for the non-existence of degree-drop spaces of some co-dimension. Based on these results, we give and efficient algorithm for constructing functions without degree-drop hyperplanes. The class of functions constructed by this algorithm, followed by an invertible affine change of variables, is very large; whether it covers all the functions without degree-drop hyperplanes is an open problem.


In Section~\ref{sec:special-cases} we examine particular classes of functions. We fully settle the cases of functions of degree $r\in\{1,2,n-2,n-1,n\}$. For  functions $f$  which are direct sums of $p$ monomials of degree $r$
 (i.e.\ the monomials have pair-wise disjoint sets of variables), we show that $\ds(f) = p-1$, which also gives a lower bound of $\ds(r,n)\ge \lfloor \frac{n-r}{r}\rfloor$ (see Theorem~\ref{thm:direct-sum}). 
We also show in Theorem~\ref{propsymrest} that symmetric functions have no degree-drop hyperplane if and only if they have even degree $r$ with $2\le r \le n-2$. 

An explicit formula for the number of functions without degree-drop hyperplane of degree $r$ in $n$ variables (i.e.\ the cardinality of $K_{1,r,n}$) is given in Section~\ref{sec:cardofK1rn}, see Theorem~\ref{thm:K1-explicit-formula}. It relies on proving that  a homogeneous function $f$ has no degree-drop hyperplanes if and only if its complement $f^c$ has no ``fast points'', i.e.\ all its discrete derivatives have degree $\deg(f^c)-1$. Counting results for the latter were given
in~\cite{SalMan17}. (Recall that the complement $f^c$ of a homogeneous polynomial $f$ is obtained by replacing each monomial $m$ of $f$ by its complement  $\frac{\prod_{i=1}^n x_i}{m}$);


Section~\ref{sec:connections-inv} examines connections between the notions we defined and existing notions. 
Specifically, $f$ has a degree-drop space $A$ of co-dimension $k$ if and only if the degree of $f1_A$ is strictly lower  than $\deg(f)+k$ (where $1_A$ is the indicator function of $A$); $f$ has a degree-drop hyperplane if and only if $\mathfrak{R}_1(f)>0$, where $\mathfrak{R}_1(f)$ is one of the affine invariants used for the classification in~\cite{lang}.

Finally, in Section~\ref{sec:experiments} we used our previous results alongside computer calculations to determine the number of degree-drop spaces for all the functions in up to 8 variables. To do so, it suffices to examine the representatives of affine equivalence classes computed by~\cite{Hou} and~\cite{lang}. We determined thus all the 8-variable functions which are optimal from the point of view of their \rds. For example, for degrees 3 and 4 in 8 variables, these are functions that maintain their degree when restricted to any subspace of codimension 2.

\section{Preliminaries}\label{2}
We denote by $\mathbb{F}_2$ the finite field with two
elements and by $\mathbb{F}_2^n$ the vector space over
$\mathbb{F}_2$ of all binary vectors of length $n$. The zero element of $\mathbb{F}_2^n$ will be denoted by $\mathbf{0}$. A function $f: \mathbb{F}_2^n \rightarrow \mathbb{F}_2$ will be called an
$n$-variable Boolean function.
There is a bijective correspondence between these functions
and the quotient ring $\mathbb{F}_2[x_1,\dots ,x_n]/(x_1^2+x_1,\dots,
x_n^2+ x_n)$ (see e.g. \cite{Carlet}). The representation in this
quotient ring (which is also an $\F_2$-vector space) is called the
{\em Algebraic Normal Form} (in brief, ANF) of the Boolean function
and has the form $f(x_1,\dots ,x_n)=\sum_{I\subseteq \{1,\dots
,n\}}a_I\prod_{i\in I}x_i$ with $a_I\in \mathbb{F}_2$.
Note that every variable $x_i$ in the ANF appears with an exponent of
at most $1$, because every element of $\mathbb{F}_2$ is its own
square.

 The {\em algebraic degree} of a Boolean function
$f$, denoted by $\deg(f)$, is the degree of its ANF. By convention,
the algebraic degree of the zero function equals $-\infty$. Those functions
of algebraic degree  at most one (respectively 2) are called {\em affine}
(respectively {\em quadratic}).

For every vector $v\in\mathbb{F}_2^n$, the Hamming weight
$\w(v)$ of $v$ is the number of its non zero coordinates ($i.e.$
the cardinality of the support set
of $v$, which is defined as $supp(v) = \{i\in \{1, \ldots, n\}:  v_i\neq 0\}$). The {\em Hamming
weight} $\w(f)$ of a Boolean function $f$ on $\mathbb{F}_2^n$ is the Hamming weight of its truth table, or equivalently
the cardinality of the support of $f$, which is defined as  $supp(f) = \{x\in\mathbb{F}_2^n:f(x)\neq 0\}$. For any set $A$, we shall denote by $|A|$ the cardinality of $A$. The system  $(e_1, \ldots, e_n)$ will be  the canonical basis of $\F_2^n$, made of all the weight 1 vectors.
Let us now give some useful notation:
\begin{notation}For every $n$-variable Boolean function $f$, we denote by $\var(f)$ the set consisting of all the elements $i\in \{1, \ldots, n\}$ such that $x_i$ appears in at least one term with nonzero coefficient in the ANF of $f$.
\end{notation}

Recall that an affine automorphism of $\mathbb{F}_2^n$ is any
mapping $\varphi$
 of the form $\varphi(x) = Mx+a$  where
$x$, $a$ and $\varphi(x)$ are viewed as column vectors, $x^T=(x_1,\ldots, x_n)$,
$a^T=(a_1,\ldots, a_n) \in \F_2^n$ and $M$ is an $n\times n$
invertible matrix over $\F_2$ (when the specific values of $M,a$ are relevant, we will also denote $\varphi$ as $\varphi_{M,a}$ or $\varphi_{M}$ if $a = \mathbf{0}$). Equivalently and more conveniently,  viewing now $x$, $a$ and $\varphi(x)$ as row 
vectors, we can define $\varphi(x) =xM^T+a$ (this notation is more usual in Boolean function theory and in coding theory, and we shall use it in the sequel).
\begin{definition}\label{deffeqcomeq} Two Boolean functions $f, g: \F_2^n \rightarrow \F_2$ are said to be affinely equivalent, which is denoted by $f\sim g$, if there
 exists $\varphi$, an  affine automorphism of $\mathbb{F}_2^n$,
such that $f=g\circ \varphi$ where $\circ$ is the operation of
composition.
\end{definition}
A parameter associated to a Boolean function is called an  affine invariant
if it is preserved by affine equivalence. For instance, the Hamming
weight and the algebraic degree are affine invariant.

The set of all $n$-variable Boolean
functions of algebraic degree at most $r$ shall be
denoted by $RM(r,n)$ (expressing its elements as their truth table binary vector of length $2^n$ provides the Reed-Muller code of order $r$).
Let $RM(r,n)/RM(r-1,n)$ be the quotient space consisting of all
cosets of $RM(r-1,n)$ in $RM(r,n)$. Unless otherwise specified, for each coset we will use as
representative the homogeneous polynomial in the coset, i.e. the
unique polynomial which contains only monomials of degree $r$.
The equivalence $\sim$ can be extended naturally to an equivalence $\sim_{r-1}$ on $RM(r,n)/RM(r-1,n)$, or more generally on any $RM(d,n)/RM(r-1,n)$ with $d\ge r$.
Namely, two  $RM(r-1,n)$-cosets are equivalent under $\sim_{r-1}$ if there is a function $f_1$ in one coset and a function $g_1$ in the other coset such that $f_1\sim g_1$. By abuse of notation we will then write $f \sim_{r-1} g$ for any function $f$ in the $RM(r-1,n)$-coset of $f_1$ and $g$ in the $RM(r-1,n)$-coset of $g_1$. Equivalently, two functions $f$ and $g$ satisfy $f \sim_{r-1} g$ if and only if there is a function $h$ such that $f \sim h$ and $\deg (g-h)\le r-1$.

For any monomial $m$
we shall denote by $m^c$ the complement $\prod_{i \in \{1,\dots n\}\setminus \var(m)} x_i$ (which also equals $\frac{x_1\cdots x_n}m$).
Moreover, if $f = \sum_{i=1}^{p} m_i$, where the $m_i$ are all
monomials of degree $r$, then we define $f^c = \sum_{i=1}^{p}
m_i^c$, also called the complement of $f$ (not to be confused with
$f+1$, which is the complement of $f$ in the sense of the Boolean operation of negation).
We recall the following result that connects the equivalence classes of $RM(r,n)/RM(r-1,n)$ and $RM(n-r,n)/(RM(n-r-1,n))$.\\
The functions
of algebraic degree 2 are fully characterized up to affine
equivalence as follows:
 \begin{lemma}\label{l2}\cite{FJ,Carlet}
Every quadratic non-affine function is affinely equivalent, under $\sim_1$, to:
$$x_1x_2+\dots +x_{2t-1}x_{2t} \textrm { for some $t\leq \frac{n}{2}$}.$$ 
\end{lemma}
\begin{proposition}(\cite[Section 4]{Hou})\label{prop:Hou-complement}
For any invertible matrix $M$ and any Boolean functions $f,g$, we have that $g = f \circ \varphi_{M}+h$ for some $h$ with $\deg(h)<\deg(f)$ if and only if $g^c = f^c \circ \varphi_{(M^T)^{-1}}+h_1$ for some $h_1$ with $\deg(h_1)<\deg(f^c)$. Hence, for any $f,g\in RM(r,n)/RM(r-1,n)$
we have $f\sim_{r-1} g$ if and only if $f^c\sim_{n-r-1} g^c$. Therefore $RM(r,n)/RM(r-1,n)$ and $RM(n-r,n)/(RM(n-r-1,n))$ have the same number of equivalence classes under the equivalence $\sim_{r-1}$ and $\sim_{n-r-1}$, respectively.
\end{proposition}
Recall that for a function $f$ and $a\in\F_2^n\setminus \{\mathbf{0}\}$, the
discrete derivative of  $f$ in the direction $a$ is defined as
$D_af(x)=f(x+a)+f(x)$. It is known that $\deg(D_a f)\leq \deg(f)-1$ (see~\cite{Lai94}). When the inequality is strict, $a$ is called a ``fast point'':
\begin{definition}\cite{Duan} Let $f:\mathbb{F}_{2}^n
\rightarrow\mathbb{F}_2$ be a non-constant Boolean function in n
variables. Any nonzero vector $a \in \mathbb{F}_{2}^n$ such that $\deg(D_af)< \deg(f) - 1$  is called a {\em fast point} for $f$.
\end{definition}

Recall that any affine hyperplane $H$ of $\F_2^n$ can be defined as the set of solutions of an equation
$\sum_{i=1}^{n} a_ix_i +a_0=0$ with $a_i\in \F_2$ and at least one of $a_1, \ldots, a_n$ being non-zero.
{\em Linear hyperplanes} are the hyperplanes with $a_0=0$.
Every affine space (also called a flat) $A$ of co-dimension $k$ is defined by
a system of $k$ linearly independent linear equations $Mx+b=0$ with $M$ being a full rank  $k\times n$ matrix over $\F_2$ and $b$ a vector in $\F_2^k$. If $b=\mathbf{0}$ then $A$ is a linear space (i.e.\ a vector subspace of $\F_2^n$).


%
Since the study of restrictions of functions to affine spaces will be the main topic of this paper, let us examine in more detail the ANF of such a restriction. The restriction of a function $f$ to an affine space $A$ will be
denoted by $f_{\mid A}$.
We can write $A=v+E$ for some vector space $E$ and some vector $v$. We then fix a basis for $E$ and use the coordinates in this basis to identify each element of $A$. This allows identifying $A$ with the set $\mathbb{F}_2^{n-k}$, where $k$ is the co-dimension of $A$, and function $f_{\mid
A}$ can then be viewed as a function in $n-k$ variables, which allows to
consider its ANF and its algebraic degree. While different choices of bases for $E$
result in different expressions for the ANF of $f_{\mid A}$, all of these functions are affinely equivalent and therefore have the same algebraic degree. \\
More precisely, we can consider the affine space $A$ as the set of solutions of $k$ affine equations, that are, after Gaussian elimination: $x_{i_1} = a_{i_1}(y), \ldots, x_{i_k} = a_{i_k}(y)$, where $i_1, \ldots, i_k$ are distinct and $a_j(y)$ are affine functions in the $n-k$ variables $\{y_1, \ldots, y_{n-k}\} = \{x_1, \ldots, x_n\} \setminus \{x_{i_1}, \ldots, x_{i_k}\}$. We obtain one of the expressions for the ANF of $f_{\mid A}$ by substituting the variables $x_{i_1}, \ldots, x_{i_k}$ with $a_{i_1}(y), \ldots, a_{i_k}(y)$ in the ANF of $f$, and we obtain a function in the remaining $n-k$ variables $\{y_1, \ldots, y_{n-k}\}$ and we consider its algebraic degree. Again, this function depends on the equations used for defining $A$ (the choice of $i_1, \ldots, i_k$ is not unique), but all choices yield functions which are affinely equivalent to each other and therefore have the same degree.
Viewing $f_{\mid A}$ as being obtained by such a substitution, it becomes clear that $\deg(f_{\mid A})\le
\deg(f)$.

\section{Degree-drop spaces: definition and basic properties}\label{sec:defs}
We introduce the following terminology:
\begin{definition}
  Let $f$ be an $n$-variable Boolean function and $A$ an affine subspace of $\mathbb{F}_2^n$.
  If $\deg(f_{\mid A})< \deg(f)$, then we call $A$ a {\em degree-drop subspace} for $f$.
\\
  The largest $k$ such that  $f$ has no degree-drop subspace of co-dimension $k$ will be called the {\em \rds } of $f$, denoted by $\ds(f)$ (with $\ds(f)=0$ if $f$ has degree-drop hyperplanes).
\end{definition}
Note that for a fixed affine space $A$, the property of $A$ to be a degree-drop space for a function $f$ depends only on the monomials
of $f$ of algebraic degree $\deg(f)$. Hence, for any polynomials
$f$ and $g$ in the same coset of $RM(r,n)/RM(r-1,n)$, we have that
$A$ is a degree-drop subspace of $f$ if and only if $A$ is a degree-drop subspace of~$g$.  It suffices therefore to study the degree-drop spaces of homogeneous polynomials.


\begin{notation}
For all integers $k,r,n$ with $1\le k\le n$ and $1\leq r\leq n$, we
denote by $K_{k,r,n}$
 the subset of $RM(r,n)/RM(r-1,n)$ consisting of the nonzero elements $f$ with the property that $\deg(f_{\mid A})= \deg(f)$
 for any affine space $A$ of co-dimension $k$, that is, admitting no degree-drop subspace of co-dimension $k$.
\\
 We denote by $\ds(r,n)$ the largest value of the \rds among all functions of degree $r$ in $n$ variables, i.e.\ the largest $k$ for which $K_{k,r,n}\neq \emptyset$.
%
 \end{notation}

 We first collect a few preliminary observations, which will be used later in the proofs of less elementary results:
\begin{lemma}\label{lem:bits-and-bobs}
Let $f$ be a homogeneous function of degree $r$ in $n$ variables and let $1\le k\le n$.\\
(i) If $H$ is a hyperplane defined by an equation involving a variable (with a nonzero coefficient) which is not in $\var(f)$, then $H$ is not a degree-drop hyperplane of $f$.\\
(ii) The hyperplane $H$ defined by the equation $x_j=0$ is a degree-drop hyperplane for $f$ if and only if $f(x_1, \ldots, x_n) = x_j g(x_1, \ldots, x_{j-1}, x_{j+1}, \ldots, x_n)$ for some homogeneous polynomial $g$ of degree $r-1$.\\
(iii) If $f$ has only one monomial in its ANF, then $f$ has $2^r-1$ degree-drop linear hyperplanes (namely all hyperplanes defined by linear equations which only contain variables that appear in $f$).\\
(iv) If $r\ge 2$ and $f = \sum_{i=1}^{p}m_i$ with $m_i$ monomials (i.e.\ $f$ has $p$ monomials in its ANF),
then $f$ has a degree-drop space of co-dimension $p$ and therefore $\ds(f)\le p-1$.\\
(v) $K_{k+1,r,n} \subseteq K_{k,r,n}$\\
(vi) If $k>n-r$ then any space of co-dimension $k$ is a degree-drop space for $f$ and $K_{k,r,n} = \emptyset$. In particular, if $f$ has degree $n$ then all spaces of any co-dimension $k\ge 1$ are degree-drop spaces for $f$ and $K_{k,n,n} = \emptyset$.\\
(vii) Any function of degree $1$ has degree-drop hyperplanes, i.e.\ $K_{k,1,n} = \emptyset$ for any $k$ and $\ds(1,n)=0$.\\
(viii) $\ds(r,n+1)\le \ds(r,n)+1$.
\end{lemma}
{\em Proof}.\\
(i) We can define $H$ by an equation of the form $x_j = a(x)$ with $j \not \in \var(f)$ and $a(x)$ being an affine function  with $j \not \in \var(a)$;
then, substituting $x_j$ with $a(x)$, as described at the end of Section \ref{2}, has no effect on the ANF of $f$ and therefore on its degree.
\\
(ii) By substituting $x_j=0$ in $f$ we obtain that $f_{|H}$ consists of those monomials of $f$ that do not contain $x_j$. Therefore $\deg(f_{|H})<\deg(f)$ if and only if there is no such monomial in $f$.
\\
(iii) is immediate.
\\
(iv) If we pick one variable $x_{j_i}$ (not necessarily distinct) from each monomial $m_i$, then the space defined by the equations $x_{j_1} =0, \ldots, x_{j_p} =0$ is a degree-drop space for $f$ and has co-dimension at most $p$.
\\
(v)  If $A$ is a degree-drop subspace of
$f$, then any subspace of $A$ is one too.
\\
(vi) If $A$ is a space of co-dimension $k$, then $f_{|A}$ being a function in at most $n-k$ variables, it cannot have algebraic degree larger than $n-k$. Therefore any space of co-dimension $k>n-r$ is a degree-drop space for any function of degree $r$.
\\
(vii) If $r=1$, the equation $f(x)=0$ defines a hyperplane $H$. The restriction of $f$ to this hyperplane is identically 0, so $f$ has $H$ as a degree-drop hyperplane.
\\
(viii) Let $k=\ds(r,n)$,
Any function $g$ of degree $r$ in $n+1$ variables can be written as $g(x_1, \ldots, x_{n+1}) = x_{n+1} h(x_1, \ldots, x_{n}) + f(x_1, \ldots, x_{n})$ with $h$ and $f$ functions in $n$ variables of degrees at most $r-1$ and $r$ respectively. We know $f$ has a degree-drop space of co-dimension $k+1$ or less; the equations that define this space (whose number is at most $k+1$), together with the equation $x_{n+1}=0$, define a degree-drop space for $g$ of co-dimension $k+2$ or less. Hence
$\ds(r, n+1)\le k+1$.
\hfill $\Box$\\

%
%

Next we examine what happens if we view an $n$-variable
function as an $(n+1)$-variable function.

\begin{lemma}\label{belkn-kn+1} Let $f$ be a homogenous function of degree $r$ in $n$ variables. We can view $f$ also as a homogeneous function in $n+1$ variables.
Then $f\in K_{k,r,n}$ if and only if $f\in K_{k,r,n+1}$,
In this sense,
$K_{k,r,n}\subseteq K_{k,r,n+1}$ and therefore $\ds(r, n) \le \ds(r, n+1)$.
\end{lemma}
{\em Proof}. Assume first that $f\in K_{k,r,n+1}$. Any affine space $A$ of $\F_2^{n}$ of co-dimension $k$ can be embedded in $\F_2^{n+1}$ as the affine space $A' = \{(x_1, \ldots, x_n,x_{n+1}): (x_1, \ldots, x_n)\in A\}$ of the same co-dimension as $A$. Since $f_{\mid A} = f_{\mid A'}$, $A$ cannot be a degree-drop space of $f$, so $f\in K_{k,r,n}$.
\\
Next assume $f\in K_{k,r,n}$ and let us look at $f$ as an
$(n+1)-$variable function. Let $A$ be an  affine space of co-dimension
$k$ in $\mathbb{F}_2^{n+1}$. If none of the equations that define $A$ contains $x_{n+1}$, then the restriction of $f$ to $A$ can be viewed in $\mathbb{F}_2^n$, and
since  $f\in K_{k,r,n}$, we have $\deg(f_{\mid A}) = \deg(f)$.
If some equation defining $A$ contains $x_{n+1}$, then we can rewrite it as
$x_{n+1} = a_{n+1}(y)$ where $y=(x_1,\dots ,x_n)\in \mathbb F_2^n$ and substituting $x_{n+1} $ with $ a_{n+1}(y)$ in $f$ does not change $f$, as $f$ does not contain $x_{n+1}$. We can then get rid of this equation in the definition of $A$, whose co-dimension becomes $k-1$. Iterating the process we are led to the case where none of the equations that define $A$ contains $x_{n+1}$ and $A$ has co-dimension $k'<k$.  Since $f\in K_{k,r,n}\subseteq K_{k',r,n}$, we have $\deg(f_{\mid A})=\deg(f)$. Hence, in all cases, we have $\deg(f_{\mid A})=\deg(f)$, which implies  $f\in K_{k,r,n+1}$. \\The rest of the proof is straightforward.
%
%
\hfill $\Box$\\

Allowing a degree-drop subspace is clearly an affine invariant property:
\begin{lemma}\label{lem:affine-invariance}
Let $f,g\in RM(r,n)/RM(r-1,n)$ be such that $f\sim_{r-1} g$, i.e.\
 $g = f \circ  \varphi + h$ for some affine automorphism $\varphi$ of $\F_2^n$
and some function $h$ with $\deg(h)<r$. Let $A$ be an affine space of $\F_2^n$ of co-dimension $k$. Then $g_{\mid_A}\sim_{r-1} f_{\mid_{\varphi(A)}}$.
Therefore, 
$A$ 
is a degree-drop space for $g$ if and only if $\varphi(A)$ is a degree-drop space for $f$. Consequently, $f\in K_{k,r,n}$  if and only if $g\in K_{k,r,n}$.
\end{lemma}
{\em Proof}.  The set $\varphi(A)$ is an affine space of co-dimension
$k$. We have that $g_{\mid_A}=(f\circ \varphi + h)_{\mid_A}=((f + h\circ \varphi^{-1})\circ \varphi)_{\mid_A}$ is affine equivalent to $(f + h\circ \varphi^{-1})_{\mid_{\varphi(A)}}=f_{\mid_{\varphi(A)}} + (h\circ \varphi^{-1})_{\mid_{\varphi(A)}}$, and $(h\circ \varphi^{-1})_{\mid_{\varphi(A)}}$ has the same algebraic degree as $h_{\mid A}$. Therefore $g_{\mid_A}\sim_{r-1} f_{\mid_{\varphi(A)}}$.

If $A$ is a degree-drop space for $g$ then $\varphi(A)$ is a degree-drop space for $f$. The converse is obvious
by observing that $f= g\circ \varphi^{-1} + h\circ \varphi^{-1}$.  
\hfill $\Box$\\

The next result shows that in order to decide whether a function has degree-drop affine spaces it suffices to consider linear spaces.
\begin{lemma}\label{lem:affine-vs-linear-dd-space} Let $f$ be an $n$-variable Boolean function and
$A=v+E$ be an affine space in  $\mathbb{F}_2^n$, where $E$ is a vector subspace and $v$ a vector. Then
$deg(f_{\mid A})=deg(f)$ if and only if $deg(f_{\mid
E})=deg(f)$. In other words, an affine space is a degree-drop space for $f$ if and only if its underlying vector space is a degree-drop space for $f$.\end{lemma}
{\em Proof}
Let $r=\deg(f)$, $A$ an affine space and $\varphi_{I,v}$ the translation by $v$. We have $f_{\mid A}\circ \varphi_{I,v}=f_{\mid E}+(D_vf)_{\mid E}$ with $deg((D_vf)_{\mid E})<r$ and $\deg(f_{\mid A}\circ \varphi_{I,v})=\deg(f_{\mid A})$. \hfill $\Box$


\section{Necessary and/or sufficient conditions for a function not to admit a degree-drop space}\label{sec:nec-suff}
\subsection{Necessary conditions}
The following necessary condition for a function to belong to
$K_{k,r,n}$ was proved in \cite{CS} (and it is easy to check):
\begin{lemma}(\cite{CS})\label{lemma:c1}
 Let $k,r,n$ satisfy $1\le k\le n$ and $1\leq r\leq n$ and let $f=\sum_{i=1}^{p}m_i$ where $m_i$ are monomials of degree $r$. If $f\in K_{k,r,n}$ then
\begin{equation}\label{c1}
\cap_{i=1}^{p}\var(m_i)= \emptyset.
\end{equation}
In particular, $f\in K_{k,r,n}$ implies $p>1$.
\end{lemma}
Lemma~\ref{lemma:c1} means that for any variable, there is at least one monomial in $f$ which does not contain it. This statement can be generalized to spaces of higher co-dimension $k$ by replacing   ``any variable" with ``any set of $k$ variables" and ``does not contain it" to ``does not contain any of these variables":
\begin{lemma}\label{cond2} Let $f$ be a homogeneous
$n$-variable Boolean function of algebraic degree $r\geq 2$ and write $f=\sum_{i=1}^pm_i$ with
$m_i$ monomials. Let $1\le k \le n$.
\\
If $f\in K_{k,r,n}$, then for any set of $k$ distinct variables
$x_{j_1},...,x_{j_k}$ there is at least one monomial in $f$ which does not contain any of the variables $x_{j_1},...,x_{j_k}$ (in other words there is an $i$ such that
$\{j_1, \ldots, j_k\}\cap \var(m_i) = \emptyset$).

\end{lemma}
{\em Proof}. Assume for a contradiction that  there is a set of  $k$ distinct variables
$x_{j_1},...,x_{j_k}$ such that each of the monomials of $f$ contains at least one
of the variables $x_{j_1},...,x_{j_k}$. Consider the affine space $A$  of
co-dimension $k$  characterized by the system of equations $x_{j_t}=0$ for all
$t=1,...,k$.
 It is easy to check that $\deg(f_{\mid_A})<\deg(f)$ contradicting the fact that $f\in K_{k,r,n}$.  \hfill $\Box$\\

The conditions in Lemmas~\ref{lemma:c1} and~\ref{cond2}  above are necessary, but not sufficient for a function to have no degree-drop spaces of a given co-dimension $k$. For example consider the function $f(x_1, \ldots, x_5) = x_1x_3+x_1x_4+x_2x_3+x_2x_4$, which satisfies condition~(\ref{c1}), yet the hyperplane defined by the equation $x_1+x_2=0$ is a degree-drop hyperplane for $f$.

\subsection{Theoretical necessary and sufficient conditions}
The two theorems below show that each of the two conditions of Lemmas~\ref{lemma:c1} and~\ref{cond2} become sufficient  when we extend them by affine equivalence.

\begin{theorem}\label{theo2} Let $f$
be a  Boolean function of algebraic degree
$r\geq 2$ in $n$ variables and let $a=(a_1,\ldots,a_n)\in \F_2^n\setminus\{ \mathbf{0}\}$ and $a_0\in \F_2$. The following two statements are equivalent:\\
(i) $f$ has a degree-drop hyperplane $H$ defined by the equation $a_0 + \sum_{i=1}^{n}a_ix_i=0$.\\
(ii) $f$ can be written as $f(x)= \left( a_0 + \sum_{i=1}^{n}a_ix_i\right)g(x)+c(x)$
for some polynomial $g(x)$ of
algebraic degree $r-1$ and  some polynomial $c(x)$ of algebraic degree at most $r-1$.
\\
Consequently, $f$ has a degree-drop hyperplane (i.e.\ $f \not \in K_{1,r,n}$)
if and only if $f(x) \sim_{r-1} x_1 g(x_2, \ldots, x_{n})$ for some
homogeneous polynomial $g$ of degree $r-1$ in $n-1$ variables.
\end{theorem}
{\em Proof}. We first prove that (i) is equivalent to (ii).\\
Assume that (ii) holds.
We have $f_{\mid {H}}=c_{\mid {H}}(x)$ meaning that $H$
is a degree-drop hyperplane for $f$.
\\
Conversely, assume (i) holds. Since $a_i\neq 0$ for at least one
value of $i\in \{1, \ldots, n\}$, we can assume without loss of generality that $a_1=1$. From Lemma~\ref{lem:bits-and-bobs} (i) we know that $x_1$ appears in the monomials of $f$ of degree $r$ (otherwise $H$ would not be a degree-drop hyperplane).
We can write $f$ as:
\[ f(x)=x_1g_1(x_2, \ldots,x_n) + g_2(x_2, \ldots,x_n),\]
with $g_1, g_2$
polynomials in $n-1$ variables,
$\deg(g_1)=r-1$, $\deg(g_2)\le r$. The hyperplane $H$ is defined by the
equation $x_1+b(x_2, \ldots,x_n)+a_0=0$, where
$b(x_2,\ldots,x_n) = \sum_{i=2}^{n}a_ix_i$. The restriction of $f$
to $H$ can be obtained by substituting $x_1$ with $b(x_2,
\ldots,x_n)+a_0$, obtaining:
\[ f_{\mid {H}}(x) = b(x_2, \ldots,x_n)g_1(x_2, \ldots,x_n) +a_0g_1(x_2, \ldots,x_n)+ g_2(x_2, \ldots,x_n) .\]
Since $H$ is a degree-drop hyperplane
we have that
\[  g_2(x_2, \ldots,x_n)=b(x_2, \ldots,x_n)g_1(x_2, \ldots,x_n)+a_0g_1(x_2, \ldots,x_n)+c(x_2, \ldots,x_n)\]
for some polynomial $c(x)$ with $deg(c)<r$ and therefore:
\begin{eqnarray*}
  f(x)&=&x_1g_1(x_2, \ldots,x_n)+ (b(x_2, \ldots,x_n)+a_0)g_1(x_2, \ldots,x_n)+c(x_2, \ldots,x_n) \\
   &=& \left( a_0 + \sum_{i=1}^{n}a_ix_i\right)g_1(x) +c(x)
\end{eqnarray*}
which means that (ii) holds.
\\
Finally, using the equivalence of (i) and (ii) and an invertible affine transformation
$\varphi$ of $\F_2^n$ such that $(\ell\circ \varphi)(x_1, \ldots,
x_n) = x_1$ where $\ell(x) =\left( a_0 + \sum_{i=1}^{n}a_ix_i\right)$ we have that $f \not \in K_{1,r,n}$ if and only if $f(x) \sim_{r-1} x_1 g(x_2, \ldots, x_{n})$ as required.
\hfill $\Box$

\begin{proposition}\label{vshyp}
For each Boolean function $f$, the set
\[
\{\mathbf{0}\}\cup\{a\in \F_2^n\setminus \{\mathbf{0}\} : \{x\in \F_2^n : a^Tx=0\} \mbox{  is a degree-drop
hyperplane of }f\}
\]
is a vector space over $\mathbb F_2$.
\end{proposition}
{\em Proof}. 
Denote $H_a = \{x\in \F_2^n:a^Tx=0\}$.
Let $a,b \in \F_2^n\setminus\{\mathbf{0}\}$ with $a\neq b$ and such that $H_a$ and $H_b$  are degree-drop hyperplanes for $f$. We have to prove that $H_{a+b}$ is also a degree-drop hyperplane for $f$. Without loss of generality, we can assume $a=e_1$ and
$b=e_2$, where $(e_1, \ldots, e_n)$ is the canonical basis of $\F_2^n$, made of all the weight 1 vectors (otherwise we can use an invertible affine transformation as
in Lemma~\ref{lem:affine-invariance}). By
Lemma~\ref{lem:bits-and-bobs} (ii), for $i=1,2$, $H_{e_i}$ is a
degree-drop hyperplane of $f$ if and only if all monomials of $f$
contain $x_i$.
 Hence $f(x_1, \ldots, x_n) = x_1 x_2 g(x_3,
\ldots, x_n)$ for some homogeneous function $g$ of degree $\deg(f)-2$. The hyperplane $H_{e_1+e_2}$ being defined by the equation
 $x_1=x_2$, one can verify that substituting $x_1$ with $x_2$ in $f$ yields $x_2 g(x_3,
\ldots, x_n)$, which has degree $\deg(f)-1$, i.e\ $H_{e_1+e_2}$ is a degree-drop hyperplane. \hfill $\Box$

Proposition~\ref{vshyp} implies that the number of linear degree-drop spaces of $f$ equals $2^j-1$ for some integer $j\ge 0$.
From Lemma~\ref{lem:bits-and-bobs} (ii), Theorem~\ref{theo2} and Proposition~\ref{vshyp} we obtain:
\begin{corollary}\label{cor:function-with-given-number-of-hyperplanes}
  A homogeneous function $f$ of degree $r$ in $n$ variables has $2^j-1$ degree-drop linear hyperplanes if and only if $f\sim_{r-1} x_1 x_2 \cdots x_j g(x_{j+1}, \ldots,x_n)$ for some homogeneous function $g$ of degree $r-j$ in $n-j$ variables which has no degree-drop hyperplanes.
\end{corollary}
Theorem~\ref{theo2} can be generalised to spaces of arbitrary co-dimension as follows:
%
\begin{theorem}\label{carc cod k}
Let $f$
be an $n$-variable homogenous Boolean
function of algebraic degree $r\le n$ and let $k\le n$. The following statements are equivalent:\\
(i) $f$ has a degree-drop affine space of co-dimension $k$.\\
(ii) $f \sim_{r-1} g$ for some homogeneous function $g$ of degree $r$ such that
each monomial of $g$ contains at least one of the
variables $x_1, x_2,...,x_k$.  \\The affine permutation $\varphi$ such that $g(x)=f\circ\varphi(x)+c(x)$ for some $c$ with $\deg(c)<r$, can be any affine transformation $\varphi$
mapping the vector space of equations $x_1=0$, $x_2=0$,..., $x_k=0$ to a degree-drop space of $f$ of co-dimension $k$.
\end{theorem}
{\em Proof}. Let us denote by $E_k$ the space defined by the equations
$x_1=0$, $x_2=0$,..., $x_k=0$ and assume that (ii) holds. The space $E_k$ is a degree-drop space for any function $g+c$ where $c$ has algebraic degree less than $r$, since we have $g_{\mid E_k}=0$. Due to the invariance under equivalence $\sim_{r-1}$ of the existence of a degree-drop space (see Lemma~\ref{lem:affine-invariance}), $f$ will also have a degree-drop space of co-dimension $k$.
\\
Now assume (i) holds, i.e.\ $f$ has a degree-drop space $A$ of co-dimension $k$. Consider the affine transformation $\varphi$
such that $\varphi(E_k)=A$. We can write $f\circ\varphi(x)$ as
 $g(x)+c(x)$ where $g$ is homogeneous of degree $r$ and
$\deg(c)<r$. From Lemma \ref{lem:affine-invariance} again, $g(x) = f\circ\varphi(x) + c(x)$ has $\varphi^{-1}(A) = E_k$ as a
degree-drop space, i.e.\ $\deg(g_{\mid E_k})<r$. In
$g_{\mid E_k}$ all the monomials of $g$  that contain at least one
of the variable $x_1,x_2,...,x_k$ will become zero after substitution. If there was a monomial of $g$ which does
not contain any of the variable $x_1,x_2,...,x_k$, it would clearly
remain unchanged in $g_{\mid E_k}$ which is a contradiction. We obtain thus (ii). \hfill
$\Box$

\begin{remark}Theorems~\ref{theo2} and~\ref{carc cod k} allow us to easily construct functions which have degree-drop hyperplanes or degree-drop spaces of co-dimension $k$. Namely, for hyperplanes, we consider functions of the form $x_1g(x_2, \ldots, x_n)$, and for affine spaces of co-dimension $k$ we can pick monomials such that all of the monomials contain at least one of the variables $x_1, \ldots, x_k$.
Moreover if we then apply affine changes of variables to the functions we constructed, we can obtain all the functions that have degree-drop hyperplanes (or spaces of co-dimension $k$, respectively).
\end{remark}

\subsection{Practical sufficient conditions and related constructions}
If we want to construct a function which does not have any degree-drop hyperplane, Theorems~\ref{theo2} and \ref{carc cod k} are not very useful. In the case of Theorem~\ref{theo2}, we would need to construct a function which is not of the form $x_1 g(x_2, \ldots, x_n)$ (this part is easy), and, moreover, it is not affine equivalent to a function of the form $x_1 g(x_2, \ldots, x_n)$; checking the latter is unfeasible (it has exponential complexity). Constructing functions without a degree-drop hyperplane of codimension $k$ using Theorem \ref{carc cod k} is even more complex.


The following results, which
under additional constraints make the conditions in Lemmas \ref{lemma:c1} and \ref{cond2} sufficient (but no longer necessary), allow to  provide families of functions which have no degree-drop hyperplanes and can be constructed efficiently.


\begin{proposition} \label{prop:card-of-monomial-intersection}
Let $f=\sum_{i=1}^pm_i$ be an $n$-variable homogeneous Boolean function of algebraic
degree $r$ with $2\le r\le n$, where the $m_i$'s are monomials.
\\
If $f$ satisfies Condition~(\ref{c1}) and the condition
\begin{equation}\label{c2}
  |\var(m_i)\cap \var(m_j)|\leq r-2 \mbox{, for all $i\neq j$ and } i, j\in\{1,\ldots,p\},
\end{equation}
then $f\in
K_{1,r,n}$ i.e.\ $f$ has no degree-drop hyperplane.
\\
More generally, for any $k < r$, if for any set of $k$ distinct variables $x_{j_1},...,x_{j_k}$, there is at least one monomial in $f$ which does not contain any of the variables $x_{j_1},...,x_{j_k}$, and if
\begin{equation}\label{c3}
  |\var(m_i)\cap \var(m_j)|\leq r-k-1 \mbox{, for all $i\neq j$ and } i, j\in\{1,\ldots,p\},
\end{equation}
then $f\in K_{k,r,n}$  i.e.\ $f$ has no degree-drop space of co-dimension $k$.
\end{proposition}
{\em Proof}.
The first part of the statement is clearly a particular case of the second part. We then prove the second part. Let $A$ be an affine space of co-dimension $k$. $A$ is characterized
by a system of $k$ independent affine equations, and without loss of generality (thanks to a Gaussian reduction and after possibly a renaming of the variables)
we can assume that these equations write $x_{i}=a_i(x_1, \ldots, x_{n-k})$, with
$i$ ranging from $n-k+1$ to $n$ and $a_i$ being affine. We obtain $f_{\mid_A}$ by using these equations to substitute $x_{n-k+1}, \ldots, x_n$.
By hypothesis, there is at least one monomial that does not contain any of the variables $x_{n-k+1},\dots ,x_n$. Without loss of generality, we can assume this is the monomial $m_p$.
Then we have ${m_p}_{\mid_A} = m_p$. Let us show that $m_p$ cannot be cancelled by any of the monomials
obtained after the substitutions from the $m_j$'s such that $j\neq p$. For any such $m_j$, we have that ${m_j}_{\mid_A}$ consists of a (possibly empty) sum of monomials which are either of degree smaller than $r$, or of degree $r$ and
containing at least $r-k$ of the variables of $m_j$ (as at most $k$ of the $r$ variables of $m_j$ have been substituted); but $|\var(m_j)\cap \var(m_p)|\leq r-k-1$  implies that each degree $r$ monomial of ${m_j}_{\mid_A}$ contains at least one variable of $m_j$ which does not appear in $m_p$.
Therefore, after the substitution, the monomial $m_p$ cannot be cancelled by any of the other monomials in $f_{\mid_A}$, so $\deg(f_{\mid_A}) = \deg(f)$.
\hfill $\Box$\\
 For homogeneous functions  having at most 3 monomials, we have a complete characterization. The case of one monomial is covered by Lemma~\ref{lem:bits-and-bobs}(iii). For $f=m_1+m_2$ with $m_i$ monomials, we know from  Lemma~\ref{lemma:c1}
and Proposition~\ref{prop:card-of-monomial-intersection} that $f$ has no degree-drop hyperplane if and only if $\var(m_1)\cap \var(m_2) = \emptyset$.
\begin{lemma} Let $f = m_1+m_2+m_3$ be an homogeneous $n$-variable Boolean
function of degree $r$ where $m_i$ are monomials and $n\ge 3$. We have that $f\in K_{1,r,n}$ if and only if $f$ satisfies condition~(\ref{c1}) and
\begin{itemize}
\item either $|\var(m_i)\cap \var(m_j)|\leq r-2$ for all $i\neq j$
\item or, after possibly permuting the indices of $m_1, m_2, m_3$, we have
$m_1 = x_{t_1}\mu(x)$, $m_2 = x_{t_2}\mu(x)$ where $t_1\neq t_2$, $\mu$ is a monomial of degree $r-1$, $\{t_1,t_2\}\cap \var(\mu) = \emptyset$ and  $\{t_1,t_2\}\nsubseteq \var(m_3)$.
\end{itemize}
\end{lemma}
{\em Proof.} In view of Lemma~\ref{lemma:c1} and Proposition~\ref{prop:card-of-monomial-intersection} we only have to prove that in the case when $f$ satisfies condition~(\ref{c1}) and two among the three monomials (say, $m_1$ and $m_2$) satisfy $|\var(m_1)\cap \var(m_2)|=r-1$, that is, $f = m_1 + m_2 + m_3$ with $m_1 = x_{t_1}\mu(x)$, $m_2 = x_{t_2}\mu(x)$, $t_1\neq t_2$, $\{t_1,t_2\}\cap \var(\mu) = \emptyset$ and $\mu$ is a monomial of degree $r-1$, the following holds: $f\in K_{1,r,n}$ if and only if $\{t_1,t_2\}\nsubseteq \var(m_3)$.
Assume first that $f \in K_{1,r,n}$. Consider the hyperplane $H$ defined by the equation $x_{t_1} = x_{t_2}$. Substituting $x_{t_1}$ in $f$ we obtain $f_{|_H} = {m_3}_{|_H}$. Then ${m_3}_{|_H}$ has degree $r$, so $m_3$ cannot contain both $x_{t_1}$ and $x_{t_2}$. Conversely, if $\{t_1,t_2\}\nsubseteq \var(m_3)$,
we can assume without loss of generality that $m_3$  does not contain $x_{t_1}$. Doing for $f$ the invertible affine change of variables which replaces $x_{t_1}$ by $x_{t_1} + x_{t_2}$ and leaves the other variables unchanged we obtain the polynomial $x_{t_1}\mu + m_3$ which
satisfies the condition in $(ii)$ (as $\var(\mu) \cap \var(m_3) = \emptyset$ by condition ~(\ref{c1})) so it belongs to $K_{1,r,n}$.
%
\hfill $\Box$

As an example of how to use Proposition~\ref{prop:card-of-monomial-intersection}, let $n$ be a multiple of $k+1$ and $r\geq k+1$. Let us start  with the monomial $x_1\cdots x_r$ and then construct iteratively monomials by taking the last $r-k-1$ variables from the current monomial, together with the next $k+1$
variables, to build the next monomial (the monomial following $x_1\cdots x_r$ being then $x_{k+2}\cdots x_{k+r}x_{k+r+1}$). We continue this process, and the indices of variables being viewed circularly in $\{1, 2, \ldots, n\}$, this leads to $\frac n{k+1}$ monomials satisfying Condition~\eqref{c3}. 
We have to then ensure that 
 $n$ is large enough so that, for any set of $k$ distinct variables $x_{j_1},...,x_{j_k}$, there is at least one monomial in $f$ which does not contain any of the variables $x_{j_1},...,x_{j_k}$.

\begin{example}\label{ex}For $(n,r,k)=(15, 5, 2)$ and $(10, 7, 1)$ respectively, we obtain:
  \begin{eqnarray*}
    f(x_1, \ldots, x_{15}) &=& x_1x_2x_3x_4x_5 + x_4x_5x_6x_7x_8 +x_7x_8x_9x_{10}x_{11}  \\
     && + x_{10}x_{11}x_{12}x_{13}x_{14} + x_{13}x_{14}x_{15}x_{1}x_{2}  \\
    f(x_1, \ldots, x_{10}) &=& x_1x_2x_3x_4x_5x_6x_7 +x_3x_4x_5x_6x_7x_8x_9 +x_5x_6x_7x_8x_9x_{10}x_1 \\
     &&+  x_7x_8x_9x_{10}x_1x_2x_3   +x_9x_{10}x_1x_2x_3x_4x_5
  \end{eqnarray*}
By Proposition~\ref{prop:card-of-monomial-intersection}, these two functions have no degree-drop space of co-dimension 2 (respectively, 1).
\end{example}

The following construction of functions with no degree-drop hyperplane generalizes\footnote{If $f=\sum_{j=1}^pm_j$ satisfies the first part of Proposition \ref{prop:card-of-monomial-intersection}, then let $i\in \var(f)$; since $f$ satisfies Condition~(\ref{c1}), then there exists $j_0\in\{1,\dots ,p\}$ such that $i\not\in \var(m_{j_0})$; for all $t\in\var(m_{j_0})$, we have $|\var(m_{j_0})\cap\var(x_im_{j_0}/x_t)|=r-1> r-2$, meaning that $x_im_{j_0}/x_t$ does not appear in $f$.} the construction in the first part of  Proposition~\ref{prop:card-of-monomial-intersection}:
\begin{theorem}\label{com.scon.ddhyp} Let $f=\sum_{j=1}^pm_j$ be a homogenous $n$-variable
function of degree $r$. If for all $i\in \var(f)$,  there exists
a monomial $m_{j_i}$ in $f$ such that: \begin{itemize}
\item $i\not\in \var(m_{j_i})$
\item for all $t\in \var(m_{j_i})$, the monomial $\frac{x_im_{j_i}}{x_t}$ is not in $f$,
\end{itemize}
then, $f$ has no degree-drop hyperplane, i.e. $f\in K_{1,r,n}$.
\end{theorem}
{\em Proof}. Assume, for a contradiction, that $f$ satisfies the conditions in the statement and
has a degree-drop hyperplane $H$ defined by the equation
$\sum_{i=1}^n c_ix_i=0$, $c_i\in \mathbb F_2$ (remember that we showed in Lemma \ref{lem:affine-vs-linear-dd-space} that we can restrict ourselves to $H$ being a linear hyperplane).
At least one coefficient is nonzero, so without loss of generality we can assume $c_1\neq 0$. The restriction of
$f$ to the hyperplane $H$ can be obtained by substituting $x_1$ by $\sum_{i=2}^{n}c_ix_i$.
The monomial $m_{j_1}$ remains unchanged in $f_{\mid H}$, as
$1\not\in \var(m_{j_1})$. Since $H$ is a degree-drop hyperplane,
$\deg(f_{\mid H})< \deg(f)$. So $m_{j_1}$ must be canceled out by a
monomial which appears when we subtituted $x_1$ in some degree
$r$ monomial $m$ of $f$ for which $1\in \var(m)$. Denoting $m'=\frac{m}{x_{1}}$ we have
$m_{\mid
H}=\sum_{i=2}^nc_ix_im'.$
Any  degree $r$
monomial of $m_{\mid H}$ is of the form $c_ix_im'$ with
$i\not\in \var(m')$ and $c_i\neq 0$, and one of them, say
$c_{i_1}x_{i_1}m'$, cancels $m_{j_1}$ out, so $c_{i_1}\neq 0$
and $m_{j_1}=x_{i_1}m'$. Therefore, $i_1\in \var(m_{j_1})$ and
$m=x_{1}m'=\frac{x_{1}m_{j_1}}{x_{i_1}}$. But by the second
condition in the theorem statement,
such a monomial $m$ cannot exist in the ANF of  $f$, which is a contradiction.
 \hfill $\Box$\\

Algorithm~\ref{alg:no-dd-hyp} is a non-deterministic algorithm to construct functions without degree-drop hyperplanes using Theorem~\ref{com.scon.ddhyp}.
\begin{algorithm}\caption{FunctionWithoutDegreeDrop$(n, r)$}\label{alg:no-dd-hyp}
\begin{algorithmic}[1]
\Require $n$, $r$ with $3\le r \le n-4$
\Ensure $F$ a set of monomials of degree $r$ in $n$ variables such that $\sum_{m\in F} m$ has no degree-drop hyperplanes
\State $F\leftarrow \emptyset$
\State $V\leftarrow$ the set of monomials of degree $r$ in $n$ variables.
\State $\Forbidden\leftarrow \emptyset$
\For {$i = 1, \ldots,n$}
\State $\Failed_i\leftarrow \emptyset$
\Repeat{}
\State $\success \leftarrow True$
\State choose $m_i\in V \setminus(\Forbidden \cup \Failed_i)$ with $i\not\in \var(m_i)$
\State $V_i \leftarrow \{\frac{m_ix_i}{x_t}: t\in \var(m_i) \}$
\If{$V_i\cap F \neq \emptyset$} 
\State $\Failed_i \leftarrow \Failed_i\cup\{m_i\}$
\State $\success \leftarrow False$
\EndIf
\Until{$\success$}
\State $F \leftarrow F\cup \{m_i\}$ 
\State $\Forbidden \leftarrow \Forbidden \cup V_i$
\EndFor
\State choose a subset $G\subseteq V\setminus\Forbidden$ 
\State $F\leftarrow F \cup G$
\State return $F$
\end{algorithmic}
 \end{algorithm}

\begin{proposition}\label{prop:alg-correctness}
    For $n\ge 10$ and $3\le r \le n-4$, as well as for $n=9$ and $r\in\{4,5\}$, Algorithm~\ref{alg:no-dd-hyp} is correct. 
    In particular, for each $i$, there is always a monomial which does not contain $x_i$ in the set  $V \setminus(\Forbidden \cup \Failed_i)$.
    The worst-case complexity of  Algorithm~\ref{alg:no-dd-hyp} is $\mathcal{O}(n^4)$ and the average complexity $\mathcal{O}(n^2)$.
\end{proposition}
\begin{proof}
Assume the algorithm execution arrived at step $i=\ell$ of the {\bf for} loop, and it is testing different candidates for $m_{\ell}$ in the {\bf repeat} loop. There are $\binom{n-1}{r}$ monomials in $V$ which do not contain $x_{\ell}$.  Let us denote by $C_{\ell}$ the number of elements in the set $\Forbidden\cup \Failed_{\ell}$ that do not contain $x_{\ell}$.  We will give an upper bound for $C_{\ell}$.

We first give an upper bound for the number of elements in the set $\Forbidden$ that do not contain $x_{\ell}$. Note that $\Forbidden = \cup_{j=1}^{{\ell}-1}V_j$ at this point. Consider each previous monomial $m_j$, with $1\le j <{\ell}$. If $m_j$ contains $x_{\ell}$, then 
$V_j$ contains $r-1$ monomials that contain $x_{\ell}$ and one that does not; if $m_j$ does not contain $x_{\ell}$, then none of the $r$ elements of 
 $V_j$  contain $x_{\ell}$. Therefore, denoting  by $n_{\ell}$ the number of monomials  among  $m_1,\ldots,m_{{\ell}-1}$ which contain $x_{\ell}$, we have that  $\Forbidden$ contains at most $n_{\ell}+({\ell}-1-n_{\ell})r$ monomials that do not contain $x_{\ell}$ (note that the sets 
$V_j$
might not be disjoint for different $j$, but we only need an upper bound).

Next we consider the cardinality of $\Failed_{\ell}$. Note that $\Failed_{\ell}$ contains the previous choices for $m_{\ell}$, which have failed because the corresponding $V_{\ell}$ satisfied  
 $V_{\ell}\cap F \neq \emptyset$. This condition means that $m_{\ell}$ was a monomial such that there was a  $t\in \var(m_{\ell})$  with $\frac{m_{\ell}x_{\ell}}{x_t} \in F$. At this point, $F = \{m_1, \ldots, m_{{\ell}-1}\}$, so there must be $1\le j <{\ell}$ such that $\frac{m_{\ell}x_{\ell}}{x_t} = m_j$. In other words, $m_{\ell} = \frac{m_jx_t}{x_{\ell}}$ for some $j<{\ell}$ such that  $m_j$ contains $x_{\ell}$ but does not contain $x_t$. For each of the $n_{\ell}$ values of $m_j$ which contain $x_{\ell}$ there are $n-r$ possibilities for $x_t$. So in total there are at most $n_{\ell}(n-r)$ possible monomials in $\Failed_{\ell}$.

 To summarise,
 \begin{eqnarray*}
    C_{\ell} &\le&  n_{\ell}(n-r) + n_{\ell}+({\ell}-1-n_{\ell})r\\ 
    &=&  n_{\ell}(n-2r+1) + ({\ell}-1)r.
\end{eqnarray*}
For fixed $n$ and $r$, we determine the values of ${\ell}$ and $n_{\ell}$
for which $n_{\ell}(n-2r+1) + ({\ell}-1)r$ is maximal (keeping in mind that
$1\le n_{\ell}\le {\ell}-1\le n-1$).
When $n-2r+1>0$, the maximum of this quantity is achieved for ${\ell}=n$ and $n_{\ell}=n-1$, and this maximum value equals $(n-1)(n-r+1)$. 
 When $n-2r+1\le 0$, the maximum of this quantity is achieved for ${\ell}=n$ and $n_{\ell}=0$, and this maximum value equals $(n-1)r$. Hence, in all cases we have $C_{\ell} \le (n-1)(n-2)$.
 It is an easy exercise to show that when $3\le r < n-4$ and $n\ge 10$ (also when $n=9$ and $r\in\{4,5\}$) we have $\binom{n-1}{r} - (n-1)(n-2)>0$, and therefore there is always at least one possible choice for $m_{\ell}$ that does not contain $x_{\ell}$ and  $m_{\ell}\in V \setminus(\Forbidden \cup \Failed_{\ell})$.

The fact that the  value of $F$ returned by the algorithm is indeed such that $\sum_{m\in F} m$ has no degree-drop hyperplane is a direct consequence of Theorem~\ref{com.scon.ddhyp}.

 For the complexity, note first that computing 
 $V_i$ 
has $\mathcal{O}(n)$ complexity.  The {\bf repeat} loop is executed at most $C_i+1$ times, i.e. at most $\mathcal{O}(n^2)$ times. Each execution involves one computation of $V_i$.
Combined with the outer {\bf for} loop being run $n$ times, gives the worse-case complexity of $\mathcal{O}(n^4)$. 

As $n$ increases,  
$\binom{n-1}{r}$ grows much faster than $C_i$ (recall $C_i\le (n-1)(n-2)$ and $3\le r\le n-4$), so the probability of trying an $m_i$ which does not satisfy the required conditions, and therefore requiring another run of the {\bf repeat} loop, are negligible. 
 The average number of times that the {\bf repeat} loop is executed can  be bounded therefore by a small constant.
This gives an average complexity of $\mathcal{O}(n^2)$.  \hfill $\Box$
\end{proof}
\begin{remark}
Note that the case $r=n-3$, $n\ge 9$ is not covered by Proposition~\ref{prop:alg-correctness} above. In this case it might be possible for the algorithm to make a series of choices for $m_1,\ldots,m_{n-1}$ for which no valid choice of $m_n$ exists. Algorithm~\ref{alg:no-dd-hyp} would need to be modified to include backtracking to the previous value of $i$ if the {\bf repeat} loop exhausts all the candidates for $m_i$ without success.
\end{remark}

 \begin{remark}
   %

 We have seen that if a function satisfies the conditions of Proposition~\ref{prop:card-of-monomial-intersection}, then it also satisfies the conditions of Theorem~\ref{com.scon.ddhyp}. The other way around is not true: see for example the function $f_{11}$ in Example~\ref{ex1} later on. 
 
 We note also that the conditions of Theorem~\ref{com.scon.ddhyp}, that are sufficient for a function to be in $K_{1,r,n}$, are not necessary. For example, we will see in Section~\ref{sec:experiments} the function $f_{22}$, which does not satisfy the conditions of Theorem~\ref{com.scon.ddhyp}. However, it is affine equivalent to one that does, so it has no degree-drop hyperplane. In fact, in Section~\ref{sec:experiments},
  we shall see that all the functions in 8 variables which have no degree-drop hyperplane are affine equivalent to functions which satisfy the conditions of Theorem~\ref{com.scon.ddhyp}. We do not know if that is still the case for larger numbers  of variables, so we propose the following open problem.
\end{remark}
\begin{open}\label{open-problem}
  Are all the functions which have no degree-drop hyperplane affine equivalent to functions satisfying the conditions of Theorem~\ref{com.scon.ddhyp}?
\end{open}

A positive answer to Open Problem \ref{open-problem} would mean that any function without degree-drop hyperplanes is among those that can be constructed efficiently  by using Algorithm~\ref{alg:no-dd-hyp}
followed by a composition with any affine automorphism.

\subsection{Relations between the existences of degree-drop spaces for functions and their complements}
Recall that if
$f=\sum_{i=1}^{p} m_i$ then $f^c=\sum_{i=1}^{p} \frac{x_1x_2\cdots
x_n}{m_i}$.
A question is whether the complement $f^c$ of a function
$f\in K_{1,r,n}$ belongs to $K_{1,n-r,n}$. \\
We observe that this is not true in
general. Indeed, the following $5$-variable function $f(x_1,x_2,x_3,x_4,x_5)=x_1x_2+x_3x_4$
belongs to $ K_{1,2,5}$ by Proposition~\ref{prop:card-of-monomial-intersection}, while
$f^c(x_1,x_2,x_3,x_4,x_5)=x_3x_4x_5+x_1x_2x_5$ does not belong to
$K_{1,3,5}$ because $f^c$ does not satisfy Condition \eqref{c1}.
Therefore we do not have necessarily $K_{1,r,n}^c\nsubseteq
K_{1,n-r,n}$ (where $K_{1,r,n}^c=\{f^c:\, f\in K_{1,r,n}\}$).
However we have the following result:

\begin{proposition}\label{pcomplem}
Let $f$ be a homogenous
Boolean function of algebraic degree $r$ in $n$ variables defined by
$f=\sum_{i=1}^{p}m_i$ with $m_i$ monomials. If $f$ satisfies Condition
\eqref{c2} 
and is such that $\cup_{i=1}^{p}\var(m_{i})=\{1,2,...,n\}$
then  $f^c\in K_{1,n-r,n}$.
\\
More generally, for any $k < r$, if $f$ satisfies
Condition 
\eqref{c3} and is such that, for any distinct variables $\{x_{i_1},...,x_{i_{k}}\}$, there exists a monomial $m_j$ such that $\{i_1,...,i_{k}\}\subseteq Var(m_j)$
then $f^c\in K_{k,n-r,n}$
\end{proposition}
{\em Proof}.
The first statement is the particular case $k=1$ of the second statement, so it suffices to prove the second statement. We will denote by $\overline{S}$ the set-theoretic complement of any set $S\subseteq\{1, \ldots,n\}$:= i.e.\ $\overline{S} = \{1, \ldots,n\}\setminus S$. Note that $\var(m^c) = \overline{\var(m)}$ for any monomial $m$.
For any $i \neq j$, using the fact that $f$ satisfies Condition~\eqref{c3}, i.e.\ $|\var(m_i)\cap \var(m_j)|\leq r-k-1$,  we have:
\begin{eqnarray*}
  |\var(m^c_i)\cap \var(m^c_j)|  &=& |\overline{\var(m_i)}\cap \overline{\var(m_j)}| \\
   &=& |\overline{\var(m_i)\cup \var(m_j)}| \\
   &=& n - |\var(m_i)\cup \var(m_j)|\\
    &=& n - \left( |\var(m_i)| +  |\var(m_j)| - |\var(m_i)\cap \var(m_j)|\right)\\
    & \le & n -2r +r-k-1\\
    &=& n-r-k-1.
\end{eqnarray*}
The condition that  for any distinct variables $\{x_{i_1},...,x_{i_{k}}\}$  there exists $m_j$ such that $\{i_1,...,i_{k}\}\subseteq Var(m_j)$ means  that for any distinct variables $\{x_{i_1},...,x_{i_{k}}\}$  there exists $m_j$ such that
$\{i_1,...,i_{k}\}\cap  Var(m_j^c)=\emptyset$ which implies that $f^c$ satisfies the conditions of Proposition~\ref{prop:card-of-monomial-intersection} for  $k$, which ends the proof.
 \hfill $\Box$

\section{Special classes of functions}\label{sec:special-cases}\label{sec5}
In this section we study the existence of degree-drop spaces for several classes of functions, namely functions of degree $2, n-2, n-1$,  functions which are direct sums of monomials 
and finally general symmetric functions.
\subsection{Degrees $2$, $n-1$,$n-2$}
For  degree $r=n-1$, we have:
  \begin{lemma}\label{lem:deg_n-1} Any $n$-variable Boolean function of degree $n-1$ has $2^{n-1}-1$ degree-drop
  linear hyperplanes. Therefore $K_{1,n-1,n}=\emptyset$ and $\ds(n-1, n)=0$.
  \end{lemma}
  {\em Proof}. Let $f$ be a (homogeneous) Boolean function of degree $n-1$. Then $f^c$ has degree 1, so it is affine equivalent to the
  function $x_1$. As recalled in Proposition~\ref{prop:Hou-complement}, two (homogeneous) functions $f$ and $g$ of degree $r$ satisfy $f\sim_{r-1} g$ if and only if their complements satisfy $f^c\sim_{n-r-1} g^c$.
We have then $f\sim_{n-2}  x_2\cdots x_n$, and by Lemma~\ref{lem:bits-and-bobs} (iii), $f$ has then  $2^{n-1}-1$ degree-drop
   linear hyperplanes.
\hfill $\Box$

\begin{example}
 The Carlet-Feng family of functions
(which are balanced Boolean functions  with optimal algebraic
immunity, see \cite{CF}
) have degree $n-1$. By Lemma~\ref{lem:deg_n-1}, any function in this family has
$2^{n-1}-1$ degree-drop linear hyperplanes.
\end{example}


For  degree $r=2$, recall from Lemma~\ref{l2} that any quadratic function $f$ in $n$ variables is equivalent under $\sim_1$ with a function of the form $x_1x_2+\dots +x_{2p-1}x_{2p}$ for some $p\ge 1$ such that $2p\le n$. The case of $p=1$ (i.e.\ one monomial) is trivial by Lemma~\ref{lem:bits-and-bobs} (iii). A part of the next proposition has been proven in \cite[Lemma 3]{CS2}, but we shall give an alternative proof.

\begin{proposition}\label{propdsmdegree2} Let $n$ and $p$ be such that  $2\le p \le \lfloor \frac{n}{2} \rfloor$. The function  of degree $2$ in $n$ variables which is the direct sum of $p$ monomials \[
  f(x_1, \ldots, x_n) = x_1 x_2 + x_{3} x_{4} + \cdots + x_{2p-1}x_{2p}
\]
has no degree-drop space of co-dimension $p-1$ but has degree-drop spaces of co-dimension $p$; hence, \ $\ds(f)=p-1$ and $\ds(2,n)= \lfloor \frac{n}{2} \rfloor - 1$. Consequently, 
\[
K_{k,2,n} =\{f: f\sim_1 x_1 x_2 + x_{3} x_{4} + \cdots + x_{2p-1}x_{2p}\mbox{ where } k+1\le p \le \left\lfloor \frac{n}{2} \right\rfloor\}.
\]
\end{proposition}
The proof requires the following lemma:
\begin{lemma}\label{lemdsmdegree2} Denote by $g_{p,n}$ the function of degree 2 in $n$ variables which is the following direct sum of $p$ monomials: $g_{p,n}(x_1, \ldots, x_n) = x_1 x_2 + x_{3} x_{4} + \cdots + x_{2p-1} x_{2p}$ with $2p\le n$. Let $f$ be a function of degree 2 in $n$ variables and let $H$ be a hyperplane in $\F_2^n$. If $f\sim_1 g_{p,n}$ then $f_{\mid H}\sim_1 g_{p-1,n-1}$ or $f_{\mid H}\sim_1 g_{p,n-1}$ (the latter only being possible if $2p\le n-1$).
    \end{lemma}
{\em Proof.}  Let $\varphi$ be the affine automorphism such that $f = g_{p,n} \circ \varphi +h$ for some function $h$ of degree at most one. By Lemma~\ref{lem:affine-invariance} we have that $f_{\mid H}\sim_1 (g_{p,n})_{\mid \varphi(H)}$. If the equation of $\varphi(H)$ contains any of the variables $x_{2p+1}, \ldots, x_n$, then, as in the proof of Lemma~\ref{lem:bits-and-bobs} (i), the ANF of $(g_{p,n})_{\mid \varphi(H)}$ is the same as the ANF of $g_{p,n}$, i.e.\ $(g_{p,n})_{\mid \varphi(H)}\sim_1 g_{p,n-1}$. Assume now that the equation of  $\varphi(H)$ contains only variables from the set $\{x_1, \ldots, x_{2p}\}$.
Without loss of generality we can assume that the equation is
$x_1=c_2x_2+c_3x_3+\cdots c_{2p}x_{2p}$ with $c_i\in \F_2$. Substituting $x_1$ in $g_{p,n}$ we obtain for the ANF of $(g_{p,n})_{\mid \varphi(H)}$:
\begin{eqnarray*}
 &&(c_2x_2+c_3x_3+\cdots c_{2p}x_{2p})x_2 + x_{3} x_{4} + \cdots + x_{2(p-1)} x_{2p} \\
  && = (x_3+c_4x_2)(x_4+c_3x_2)+\cdots +(x_{2p-1}+c_{2p}x_2)(x_{2p}+c_{2p-1}x_2) +cx_2
\end{eqnarray*}
for some constant $c\in \F_2$. By making the affine transformation $x_1 \leftarrow x_3+c_4x_2$,
$x_2 \leftarrow x_4+c_3x_2$,....,$x_{2p-2} \leftarrow x_{2p}+c_{2p-1}x_2$, $x_{2p-1}\leftarrow x_2$ this function is equivalent under $\sim_1$ to $g_{p-1,n-1}$.
\hfill$\Box$\\

Let us now prove Proposition \ref{propdsmdegree2}.\\
{\em Proof.} The fact that $f$ has a degree-drop space of co-dimension $p$ follows from Lemma~\ref{lem:bits-and-bobs}  (iv). Let us then prove that $f$ has no degree-drop space $A$ of co-dimension $p-1$.
 We can write $A$ as an intersection of $p-1$ hyperplanes, $A = \cap_{i=1}^{p-1} H_i$. Define $f_0=f$
and $f_i = (f_{i-1})_{\mid H_{i}}$ for $i = 1, \ldots, p-1$.
We have $f_{p-1} = f_{\mid A}$. Using the notation and results of Lemma~\ref{lemdsmdegree2}, it follows by induction that for all $i = 0, \ldots, p-1$ we have $f_i \sim_1 g_{t_i, n-i}$ where $t_0=p$ and $t_i\in \{t_{i-1}, t_{i-1}-1\}$. Therefore $t_{p-1}\ge 1$, which implies that $f_{p-1}$, and therefore $f_{\mid A}$,  has degree 2, so $A$ is not a degree-drop space of $f$.
\hfill$\Box$\\

 For  degree $r=n-2$, we have:
\begin{proposition}\label{casrnmo2}Let $n\geq 4$ be an integer. We have:
\begin{itemize}
\item
 if $n$ is odd, then $K_{1,n-2,n}=\emptyset$ and therefore $\ds(n-2,n) = 0$,
 \item if $n$ is even, then
 $$K_{1,n-2,n}=\{f: \textrm{
 $f\sim_{n-3} g$ where $g^c(x_1, \ldots, x_n)\sim_{1} x_1x_2+\dots +x_{n-1}x_{n}$ }\}$$  
 and $K_{2,n-2,n}=\emptyset$. Therefore $\ds(n-2,n) = 1$.
\end{itemize}
\end{proposition}
{\em Proof}.
From the results recalled in Lemma~\ref{l2} and  Proposition~\ref{prop:Hou-complement},
any polynomial $f$ of degree $r=n-2$ is such that
$f \sim_{n-3} g$ and $g^c(x_1, \ldots, x_n) = x_1x_2+\dots +x_{2t-1}x_{2t}$ with $t\geq 1$ and $2t\le n$.
If $2t<n$ then all the monomials of $g$ contain $x_n$, so by Lemma~\ref{lemma:c1} $f$ is not in $K_{1,n-2,n}$.
If $2t=n$ then there is no variable that appears in all the monomials of $g$ (since $n\ge 4$) and
any two monomials in $g$ have exactly $r-2$ variables in common. According to Proposition~\ref{prop:card-of-monomial-intersection}, we deduce
$f \in  K_{1,n-2,n}$. One can check that any monomial of $g$ contains either $x_1$ or $x_n$, so by Lemma~\ref{cond2} we have $K_{2,n-2,n}= \emptyset$.
 \hfill $\Box$\\

 \subsection{ Direct sums of monomials}
Before giving the general result  regarding the direct sum of monomial functions, we shall need the following definition and lemmas.

Recall that the rank, denoted by ${\rm rank}(f)$, of a function $f$ in $n$ variables is the minimum integer $n_1$ such that there is a function $g$ which depends on $n_1$ variables and $g\sim f$. This notion can be extended to the quotient
$RM(r,n)/ RM(r-1,n)$, by defining the rank of an element $f\in RM(r,n)/ RM(r-1,n)$, denoted by $\rank(f)$,
    as the minimum integer $n_1$ such that there exists $g\in RM(r,n)/ RM(r-1,n)$ such that $f\sim_{r-1} g$ and $g$ depends on $n_1$ variables (that is, is independent of $n-n_1$ variables, which means that $D_{e_i}g$ is identically zero for $n-n_1$ vectors $e_i$ of the canonical basis
    ).
Obviously, $\rank$ is invariant under $\sim_{r-1} $ (due to the transitivity of $\sim_{r-1} $). Note also that in general, for a function $f$ of degree $r$, ${\rm rank}(f)$ can be different from  $\rank(f)$. Indeed, the set of all $g$ such that $f\sim g$ is included in the set of
all $g$ such that $f\sim_{r-1} g$, implying that $\rank(f)\leq {\rm rank}(f)$, and the inequality can be strict; for instance, for $n=5$ and $f(x)=x_1x_2+x_3x_4+x_5$, we have ${\rm rank}_{1}(f)=4$ and ${\rm rank}(f)=5$.
\begin{lemma}\label{lem:rank-direct-sum}
    Let $f$ be a homogeneous function of degree $r$ in $n$ variables with $2\le r \le n$. If $f$ has the property that for any two distinct monomials $m, m'$ of $f$  we have $|\var(m) \cap \var(m')|\le r-2$,  then $\rank(f)= |\var(f)|$. In particular, if $f$ is a direct sum of $p$  degree $r$ monomials, then $\rank(f)= pr$.
\end{lemma}
{\em Proof.}
It was proven in~\cite[Corollary~2]{SalMan17} that a function $f$ of degree $r$ in $n$ variables has no fast points if and only if $\rank(f)=n$. So it suffices to show that our function $f$, as a function in $n=|\var(f)|$ variables, has no fast points. Let $a=(a_1, \ldots, a_{n})\neq \mathbf{0}$  and consider the derivative $D_a f$. The derivative of a monomial $x_{i_1}\cdots x_{i_r}$ of $f$ in direction $a$ equals
\[
 \prod_{j=1}^{r} (x_{i_j} + a_{i_j}) - \prod_{j=1}^{r} x_{i_j} = \sum_{j=1}^{r}a_{i_j}  x_{i_1}\cdots x_{i_{j-1}}x_{i_{j+1}}\cdots x_{i_r}+h(x)
\]
where $h(x)$ is a polynomial of degree at most $r-2$. Note that the monomials of degree $r-1$ in this derivative (if there are any) involve $r-1$ of the variables from the original monomial. The assumption that for any two monomials $m, m'$ of $f$ we have $|\var(m) \cap \var(m')|\le r-2$ implies that none of the monomials of degree $r-1$ in the derivative of $m$ can be cancelled by monomials in the derivative of $m'$. Since the vector $a$ is non-zero, it has at least one non-zero component, so there is at least one monomial $m$ in $f$ whose derivative has degree $r-1$.
 Therefore  $\deg(D_a f) = r-1$ regardless of the value of $a$, so there are no fast points.
\hfill$\Box$\\

The behavior of $\rank(f)$ when $f$ is restricted to an affine space is examined in the following:
\begin{lemma}\label{lemdsmrk}Let $f$ be a degree $r$ Boolean  function in $n$ variables and let $A$ be an affine subspace of $\F_2^n$ of co-dimension $k$. Then $\rank(f_{\mid A})\leq \min(n-k, \rank(f))$.
\end{lemma}
{\em Proof.}
Note that $\rank(f_{\mid A})\leq n-k$ is obvious as $f_{\mid A}$ is a function in $n-k$ variables.
Denote $n_1=\rank(f)$ and let $g$ be a function in $n_1$ 
variables such that $f\sim_{r-1} g$. If  $\varphi$ is an affine
automorphism such that $g=f\circ \varphi +h$ for some $h$ of 
degree at most $r-1$, then, by Lemma~\ref{lem:affine-invariance}, $f_{\mid A}\sim_{r-1} g_{\mid \varphi^{-1}(A)}$. 
The restriction of a function $g$ in  $n_1$ variables to any 
affine space $B$ has $\rank$ at most $n_1$. Indeed, assume 
$g=m_1+\dots+m_p +t$, where the $m_i$'s are monomials of degree 
$r$ and $\deg(t)<r$.  Let $B$ be defined by $k$ linearly 
independent equations. If an equation of $B$ contains a 
variable which is not in $g_1$, then this equation does not 
affect $g_1$ if we substitute the expression of this variable 
(from that equation) in $g$. Then, the only equations that can 
affect $g_1$ are those whose variables  are all included in 
$\var(g_1)$ and the number of variables in $g_1$ will then not
increase when restricting $g$ to $B$ that is,  
$\rank
(g_{\mid B})\leq n_1$. Therefore
 $\rank(f_{\mid A})=\rank(g_{\mid \varphi^{-1}(A)})\leq n_1 =\rank (f)$.
\hfill$\Box$\\

We can now state our result on direct sums of monomials, which is one of our main results, generalizing Proposition \ref{propdsmdegree2} to all degree $r\geq 2$.
\begin{theorem} \label{thm:direct-sum}
Let $2\le r <n$ and $2\le p \le \lfloor \frac{n}{r} \rfloor$. The function  of degree $r$ in $n$ variables which is the direct sum of $p$ monomials
\[
  f(x_1, \ldots, x_n) = x_1 x_2\cdots x_r + x_{r+1} x_{r+2} \cdots x_{2r} + \cdots + x_{(p-1)r+1} x_{(p-1)r+2} \cdots x_{pr}
\]
has no degree-drop space of co-dimension $p-1$ but has degree-drop spaces of co-dimension $p$, i.e.\ $\ds(f)=p-1$. Consequently $\ds(r,n)\ge \lfloor \frac{n-r}{r} \rfloor$.
\end{theorem}
{\em Proof}. By Lemma~\ref{lem:bits-and-bobs} (iv), $f$ has degree-drop spaces of co-dimension $p$.
We need to prove that no affine space $A$ of co-dimension $p-1$ can be a degree-drop hyperplane.
When $p\le r$, this follows from Proposition~\ref{prop:card-of-monomial-intersection}.
%
Assume now $p>r$. The case $r=2$ was proven in Proposition~\ref{propdsmdegree2}, so we can assume $r\ge 3$.
Assume, for a contradiction, that $A$ is a degree-drop subspace for $f$.
Let the equations that define $A$ be  (in diagonalized form)
$x_{i_1} = a_{i_1}(y), \ldots, x_{i_{p-1}} = a_{i_{p-1}}(y)$, where $i_1, \ldots, i_{p-1}$ are distinct and $a_{i_1}(y), \ldots, a_{i_{p-1}}(y)$ are affine functions in the $n-(p-1)$ variables $\{y_1, \ldots, y_{n-(p-1)}\} = \{x_1, \ldots, x_n\} \setminus \{x_{i_1}, \ldots, x_{i_{p-1}}\}$.
To compute $f_{\mid A}$ we substitute the variables $x_{i_1}, \ldots, x_{i_{p-1}}$ with $a_{i_1}(y), \ldots, a_{i_{p-1}}(y)$ respectively. We partition the monomials of $f$,
writing $f = f_0 + f_1 + f_2$ where the monomials in 
$f_0$ have no  variable to be substituted, the monomials in $f_1$ have exactly one variable that will be substituted and the monomials in $f_2$ have two or more variables that will be substituted. Let $p_0, p_1, p_2$ be the number of monomials in $f_0, f_1, f_2$ respectively. Obviously $p_0+ p_1+ p_2 = p$.
There are $p-1$ variables to be substituted, $p_1$ of them are in $f_1$ and at least $2p_2$ are in $f_2$; therefore $p-1\ge 2p_2+p_1$. Using $p_0+ p_1+ p_2 = p$, this implies $p_0>p_2$.

As $f$ contains $p$  monomials and no variable appears in more than one monomial, there must exist at least one monomial $m$ which does not contain any of $x_{i_1},...,x_{i_{p-1}}$ (i.e.\ $m$ is a monomial of $f_0$) and therefore  remains unchanged after substitution.  To cancel $m$ there must be at least one other monomial $m'$ of $f$ such that all the variables of $m'$ are substituted (therefore $m'$ is a monomial of $f_2$). In other words, $p_0 \neq 0$ and $p_2\neq 0$. Obviously $f_{\mid A} = (f_0)_{\mid A} + (f_1)_{\mid A} + (f_2)_{\mid A}$. Since we assumed $A$ is a degree-drop space for $f$, we have $\deg(f_{\mid A})<r$, which means $f_{\mid A} \sim_{r-1} 0$ and therefore 
\begin{equation}\label{eq:proof-dsm}
  \rank((f_0)_{\mid A} + (f_1)_{\mid A}) = \rank(f_{\mid A} + (f_2)_{\mid A}) = \rank((f_2)_{\mid A}).
\end{equation}
Since no variables are substituted in $f_0$ and only one variable is substituted in each monomial of $f_1$, the monomials of $f_0$ do not cancel out in $(f_0)_{\mid A} + (f_1)_{\mid A}$, which means
$\var(f_0)\subseteq \var((f_0)_{\mid A} + (f_1)_{\mid A})$. Moreover, any two monomials in $(f_0)_{\mid A} + (f_1)_{\mid A}$ have at most one variable in common (since the unique substitution in each monomial of $(f_1)_{\mid A}$ gives rise to a number of monomials each of which contains at most one new variable) and therefore at most $r-2$ variables in common (since $r\ge 3$). We apply Lemma~\ref{lem:rank-direct-sum} to $(f_0)_{\mid A} + (f_1)_{\mid A}$, obtaining $\rank((f_0)_{\mid A} + (f_1)_{\mid A}) = |\var((f_0)_{\mid A} + (f_1)_{\mid A})| \ge |\var(f_0)|  = rp_0$. On the other hand, by Lemma~\ref{lemdsmrk}, $\rank( (f_2)_{\mid A}) \le \rank(f_2) = rp_2$. Combining these inequalities with~\eqref{eq:proof-dsm}, we obtain $rp_0 \le rp_2$, i.e. $p_0 \le p_2$, contradicting the inequality $p_0>p_2$ proven at the end of the previous paragraph.
\hfill $\Box$

\begin{corollary}\label{cor:bounds-deg-stab}
 We have the following lower and upper bounds for $\ds(r,n)$ when $2\le r\le n-1$:
\[
\left\lfloor \frac{n-r}{r}\right\rfloor \le \ds(r,n) \le n-r-1.
\]
When $r=2$, equality is achieved for the lower bound.
Slightly improved upper bounds hold in the following cases:\\
(i) if $r\in \{3,4\}$ and $n\ge 8$, then $\ds(r,n) \le n-6$.\\
(ii) if $r$ is odd and $5\le r\le n-2$, then $\ds(r,n) \le n-r-2$.\\
Equality is achieved for the upper bounds when $n=8$, or when $r\in \{n-1,n-2\}$.
\end{corollary}
{\em Proof.}
The lower bound follows from Theorem~\ref{thm:direct-sum} above. For the upper bound, Lemma~\ref{lem:bits-and-bobs} (viii) implies $\ds(r,n)\le \ds(r,n-t)+t$ for any $t$ such that $r\leq n-t$, that is, $1\le t\le n-r$. Recall from the beginning of Section \ref{sec5} that $\ds(r,r)=\ds(r,r+1)=0$. Hence, for $t=n-r-1$, we obtain a better bound than for $t=n-r$: $\ds(r,n)\le \ds(r,r+1)+n-r-1$, which gives $\ds(r,n)\le  n-r-1$ since $\ds(r,r+1)=0$ by Lemma~\ref{lem:deg_n-1}. Further, from Proposition~\ref{casrnmo2}, $\ds(r,r+2)=0$ when $r$ is odd and $r\leq n-2$ which implies now that  for $t=n-r-2$, we obtain a better bound than for $t=n-r-1$. This yields $\ds(r,n)\le  n-r-2$. Finally, the computer calculations described in Section~\ref{sec:experiments} for $n=8$ give $\ds(3,8) = \ds(4,8) =2$ (see Table~\ref{table-ds}), which combined with $\ds(r,n)\leq \ds(r,8)+n-8$ yield the upper bounds at point (i).
\hfill $\Box$

\subsection{ Symmetric functions} 
A function is called symmetric if it is invariant under every permutation
of the variables, i.e. $f(x_1, \ldots, x_n) = f(x_{\pi(1)}, \ldots,
x_{\pi(n)})$ for any permutation $\pi$ on $\{1, \ldots, n\}$. For Boolean functions, this is equivalent to the fact that $f$ outputs
the same value for all the inputs of a same Hamming weight. Further,
when a symmetric Boolean function contains a degree $r$ monomial,
then it contains all the other degree $r$ monomials. Examples of symmetric functions are the
majority functions (and more generally, threshold
functions) which are known for having optimal algebraic immunity, an
important notion in cryptography (see \cite{dalai,Carlet}).

The restriction of a symmetric function to an affine space is not necessarily a symmetric function. For example, the symmetric homogeneous function of degree 3 in 5 variables, when restricted to the space defined by $x_1 = x_2+x_3$ becomes $x_2x_3x_4 + x_2x_3x_5 + h(x_2,x_3,x_4)$ for a function $h$ of degree 2, so it is no longer symmetric.

In the next theorem, we assume $2\leq r\leq n-2$ since the cases of $r=1,n$ were addressed in Lemma \ref{lem:bits-and-bobs} (vii) and the case $r=n-1$ was addressed in Lemma~\ref{lem:deg_n-1}.

\begin{theorem}\label{propsymrest}
Let $f$ be a symmetric Boolean function in $n$ variables of degree~$r$, with 
$2 \leq r\leq n-2$.
\\
(i) If  $r$ is even,
then $f$ has no degree-drop hyperplane. 
\\
(ii) If $r$ is odd, 
then $f$ has exactly one
degree-drop linear hyperplane,  of equation
$x_1+x_2+
\ldots + x_n =0$.
\end{theorem}
{\em Proof}.
Recall that, thanks to Lemma \ref{lem:affine-vs-linear-dd-space}, we can restrict ourselves to linear hyperplanes.
Since $f$ is
invariant under permuting the variables, we have that $\sum_{i=1}^{n}a_ix_i=0$
defines a degree-drop hyperplane if and only if $x_1+ x_2+ \ldots +
x_{\omega}=0$ defines a degree-drop hyperplane, where $\omega = \w((a_1, \ldots, a_n))$.
So let $H$ be a hyperplane with equation
$x_1+x_2+...+x_{\omega}=0$. By substituting $x_1$ by $x_2+...+x_{\omega}$
in $f$, we obtain $f_{\mid H}$ as a polynomial in the $n-1$
variables $x_2, \ldots, x_n$.

If $\omega =1$, $x_1$ is substituted by $0$; obviously all the monomials of degree $r$ of $f$
which contain $x_1$ will no longer appear in $f_{\mid H}$, whereas
those which do not contain $x_1$ (and there will be such monomials
as $r<n$) will appear in $f_{\mid H}$. Therefore, in this case $H$
is not a degree-drop hyperplane.

Next assume $\omega \geq 2$. Consider a monomial $m$ of $f$ of degree
$r$
which does not contain $x_1$ (such monomials exist as we assumed
$r<n$). 
Following the substitution of $x_1$ in $f$, the monomial $m$ is obtained
from $m$ itself and from any monomial of the form $x_1\frac{m}{x_i}$
where $i\in \var(m)\cap \{2, \ldots, \omega\}$. Hence $m$ appears
in $f_{\mid H}$ if and only if $\var(m)\cap \{2, \ldots, \omega\}$
has an even number of elements.

If $\omega = n$, i.e.\ $H$ is defined by the equation $x_1+x_2+
\ldots + x_n =0$,
then $\var(m)\cap \{2, \ldots, \omega\} = \var(m)$ has $r$ elements.
Therefore, if $r$ is odd then none of the monomials of degree $r$
appears in $f_{\mid H}$, whereas if $r$ is even then all the
monomials of degree $r$ will appear in $f_{\mid H}$. Hence $H$ is a
degree-drop hyperplane if and only if $r$ is odd.

Finally consider the case when $2\le \omega \le n-1$ (and $r\le n-2$).
We will show that in this case there are monomials $m$ for which the
set $\var(m)\cap \{2, \ldots, \omega\}$ has an even number of
elements (and there are also monomials $m$ for which this set has an
odd number of elements, but this plays no role). Namely, since $r \le n-2$, there exists at
least one integer $u$ such that $2\le u \le \omega \le u+r-1\le
n-1$ (it suffices indeed to take $u=\max(\omega-r+1, 2)$: if $2\geq \omega-r+1$,  then $u=2$ and we have clearly
  $u+r-1\geq \omega\geq u$ and since $r \leq  n - 2$, we have $u+r-1=r+1\leq n-1$; and if $2\leq \omega-r+1$ then $u=\omega-r+1$ and we have $u\leq \omega$ because $r\geq 2$ and $u+r-1=\omega \leq n-1$).  Consider the following two monomials of degree $r$: $m_1 =
x_u x_{u+1}\cdots x_{u+r-1}$ and $m_2 = x_{u+1} x_{u+2}\cdots
x_{u+r}$. Note that $|\var(m_1)\cap \{2, \ldots, \omega\}| =
|\var(m_2)\cap \{2, \ldots, \omega\}|+1$. Therefore, according to the property proved above, one and only one among the
monomials $m_1$ and $m_2$ will appear in $f_{\mid H}$ . In other words, $f_{\mid H}$ will contain at least one
monomial of degree $r$, therefore $H$ cannot be a degree-drop
hyperplane in this case.
\hfill $\Box$
\begin{example}
  The majority function returns the value of the majority of its $n$ input bits, returning 0  in case of equal number of 0 and 1 in the input when $n$ is even.
  It is obviously a symmetric function. Its degree is $2^{\lfloor \log_2 n\rfloor}$ (see~\cite{dalai}
). By Theorem~\ref{propsymrest}, we see that the majority function has no degree-drop hyperplane except for the cases where $2^{\lfloor \log_2 n\rfloor}=n$ and $2^{\lfloor \log_2 n\rfloor}=n-1$, that is $n$ is of the form $2^s$ or $2^s+1$ for some integer $s\ge 0$.
\end{example}
When a symmetric function admits a degree-drop hyperplane, we can also look whether the degree drops by 1 or by more than 1, after
the restriction to some hyperplanes. Unlike Theorem \ref{propsymrest}, in which the
degree of a function  drops in an affine hyperplane  if and only if
its degree drops in the corresponding linear hyperplane (by using
Lemma~\ref{lem:affine-vs-linear-dd-space}), we see in the next proposition that the way
the degree drops in an affine hyperplane is not necessarily the same
as in its corresponding
linear hyperplane.  \\Note that we need to take into account the whole function and not only its degree $r$ part.
\begin{proposition}
Let $f$ be a symmetric Boolean function in $n$ variables of degree
$r\geq 2$ and let $H$ be a degree-drop linear hyperplane of $f$, defined by the equation $x_{i_1} + \cdots + x_{i_{\omega}} = 0$ for some $1\le i_1 < \ldots <i_{\omega}\le n$. Denote by $H'$ the hyperplane of equation $x_{i_1} + \cdots + x_{i_{\omega}} = 1$.\\
 (i) If  $f$ contains a degree $r-1$ monomial, then the degree of $f$ drops by 1 when $f$ is restricted to $H$; it drops by at least 2 when $f$ is restricted to $H'$.\\
(ii) If $f$ contains no degree  $r-1$ monomial, then the degree drops by  1 when $f$ is restricted to $H'$, and by at least 2 when $f$ is  restricted to $H$, unless $r=n-1$ and $\omega<n$ is even (then the degree drops by  1, both when $f$ is restricted to $H'$ and when $f$ is  restricted to $H$) and $r=n$ and $\omega$ is even (then the degree drops by  1 when $f$ is restricted to $H$, and by at least 2 when $f$ is  restricted to $H'$).
\end{proposition}
{\em Proof}.
As in the previous proof, due to symmetry it suffices to  consider the hyperplanes $H$ and $H'$ given by the equation $x_1+x_2+...+x_{\omega}=\epsilon$ with $\epsilon=0,1$ respectively. We substitute $x_1 = x_2+...+x_{\omega}+\epsilon$ (which writes $x_1=\epsilon$ if $\omega=1$).
\\
(i) Assume a degree $r-1$ monomial $t$ appears in $f$. Since $f$ is
symmetric, this means that all degree $r-1$ monomials appear in $f$. Consider such a monomial $t$ that does not contain $x_1$. Then $t$ is invariant under the substitution of $x_1$ and $t$ cannot be obtained from other monomials not containing $x_1$,  after substitution.  Let us determine how many times $t$ can be obtained after substitution from monomials containing $x_1$: \\
- the degree $r$ monomial $x_1t$ gives $t$ for each $i\in \var(t)\cap \{2,3,...,\omega\}$, and once more if $\epsilon =1$;  \\
- the degree $r-1$ monomial $x_1\frac{t}{x_i}$ gives $t$ for each  $i\in \var(t)\cap
\{2,3,...,\omega\}$. \\
These are the only monomials giving $t$ (not forgetting that $t$ itself gives $t$ after substitution). Then, $t$ appears the odd number of times $1+2|\var(t)\cap
\{2,3,...,\omega\}|$ after the substitution if $\epsilon=0$ and the even number of times $2+2|\var(t)\cap
\{2,3,...,\omega\}|$ if $\epsilon=1$, therefore $t$ appears in  $f_{|H}$ but not in $f_{|H'}$.
%
\\
(ii) Assume  now that no degree $r-1$ monomial appears in $f$. A
monomial $t$ of degree $r-1$ that does not contain $x_1$ can appear
in $f_{\mid H}$ after applying the substitution to $x_1t$ in $f$. It
will appear  once for each $i$ in $\var(t)\cap \{2,3,...,\omega\}$, and one extra time if $\epsilon =1$. This means
that $t$ appears in $f_{\mid H}$ if and only if $\var(t)\cap
\{2,3,...,\omega\}$ has an odd number of elements;  $t$ appears in $f_{\mid H'}$ if and only if $\var(t)\cap
\{2,3,...,\omega\}$ has an even number of elements.
\\
a)
Assume $r\leq n-2$. From Theorem~\ref{propsymrest}, $f$ has a degree-drop linear hyperplane if and only if $r$ is odd and the hyperplane  is
$x_1+x_2+...+x_n=0$ meaning that $\omega=n$. Since we assume that $f$ has a degree-drop hyperplane, $r$ is odd and $\omega=n$. Since $|\var(t)\cap
\{2,3,...,n\}|=|\var(t)|=r-1$ is even, then $f_{\mid H}$ contains no
degree $r-1$ monomial, that is $H$ is
a degree-drop space by at least 2; on the other hand, $t$ appears in $f_{\mid H'}$, so the degree drops by 1 when $f$ is restricted to $H'$.
\\
b) Assume $r=n-1$. From Theorem~\ref{propsymrest} we know that $H$ is a degree-drop hyperplane if and only if $\omega$ is even. Consider first the case  $\omega<n$.
Since $r = n-1$, consider the following degree $n-2$ monomials: $t_1 =
x_2 x_3\cdots x_{n-1}$ and $t_2 = x_3 x_4\cdots x_n$ which do not
contain $x_1$. Clearly,   after substitution, $t_i$ appears  $|\var(t_i)\cap \{2,3,...,\omega\}|$ times, with $i=1,2$, since $t_i$ can come only from $x_1t_i$ (of degree $n-1$), because in the hypothesis of (ii), we have that no degree $r-1$ monomial appears in $f$. But $|\var(t_1)\cap \{2,3,...,\omega\}|=|\var(t_2)\cap
\{2,3,...,\omega\}|+1$  means that exactly one of the monomials $t_1$
and $t_2$ will appear in $f_{\mid H}$ (and similarly one of them will appear in $f_{\mid H'}$), and therefore the degree drops by 1 in this case. It remains to consider the case $\omega=n$, which means that $n$ is even, as $\omega$ is even. Any monomial $t$ of degree $r-1$ which does not contain $x_1$ will appear $|\var(t)\cap \{2,3,...,n\}| = r-1 = n-2$ times in $f_{\mid H}$, and $r = n-1$ times in $f_{\mid H'}$. Since $n$ is even, the degree drops by 1 on $H'$ and by at least 2 on $H$.
\\
 c) For $r=n$, we know from Theorem~\ref{propsymrest} that all the hyperplanes are degree-drop hyperplanes of $f$ and we observe that there is only one
 degree $r-1=n-1$ monomial $t$ which does not contain $x_1$, namely $x_2\cdots x_n$, for which  $t$ appears  $|\var(t)\cap \{2,3,...,\omega\}| = \omega-1$ times in $f_{\mid H}$ and $\omega$ times in $f_{\mid H'}$.\\
 This completes the proof since in all the other cases, there is no degree drop.
%
 \hfill $\Box$
 \section{Determining the cardinality of
$K_{1,r,n}$ via connections to fast points}\label{sec:cardofK1rn}
We start this section by showing that the existence of degree-drop hyperplanes for a homogeneous function $f$ is equivalent to the existence of fast points for its complement function $f^c$. We then use this result to, firstly, reformulate Theorem~\ref{com.scon.ddhyp} for giving a sufficient condition for a function to have no fast point, and secondly, for counting the number of functions that have no degree-drop hyperplane, i.e.\  the
cardinality of the set $K_{1,r,n}$, using the results from~\cite{SalMan17,SalOzb20} on the number of functions admitting fast points.

\begin{remark} A comparison between the properties of having a degree-drop hyperplane
and having a fast point was made in~\cite{CS}. It was shown that there
exist functions that have no fast point but do have degree-drop
hyperplanes (e.g.\ $f(x_1,\dots ,x_5)= x_1x_2x_5+x_3x_4x_5$), and
there exist functions that do have fast points but have no degree-drop hyperplane (e.g.\ $f(x_1,\dots ,x_5)= x_1x_2+x_3x_4$).
There are also functions which have neither; for example, the function $f = \sum_{i=1}^{p}m_i$ with $m_i$ monomials satisfying Conditions (\ref{c1}), (\ref{c2})  and $ \bigcup_i\var(m_i)=\{1,2, \ldots, n\}$, see Propositions~\ref{prop:card-of-monomial-intersection} and~\ref{pcomplem}.
\end{remark}
We show that the existence of degree-drop hyperplanes for a homogeneous function $f$ is equivalent to the existence of fast points for the complement function $f^c$.
 We will actually prove a more general result. \\In the next theorem, we use the notation $a\cdot x$ for the usual inner product $a\cdot x=\sum_{i=1}^{n} a_ix_i$ between $a=(a_1,\dots ,a_n)\in \mathbb F_2^n$ and $x=(x_1,\dots ,x_n)\in \mathbb F_2^n$. We shall need to consider several vectors. We shall denote them by $a^{(1)},\dots , a^{(k)}$, to avoid any confusion with the coordinates $a_1,\dots ,a_n$ of a vector.
\begin{theorem}\label{thm:dd-space-fast-space}
  Let $f$ be a homogeneous Boolean function of algebraic degree $r$ in $n$ variables, with $r<n$. Let $1\le k\le n-r$, and let $a^{(1)}, \ldots, a^{(k)}$ be $k$ linearly independent elements of $\F_2^n$. The following statements are  equivalent:
\begin{itemize}
  \item The linear space defined by the $k$ equations $a^{(1)}\cdot x =\ldots =a^{(k)}\cdot x =0$
  is a  degree-drop subspace for $f$.
  \item
  We have
  $\deg(D^{(k)}_{a^{(1)}, \ldots, a^{(k)}} f^{c}) <\deg(f^c)-k$, where $$D^{(k)}_{a^{(1)}, \ldots, a^{(k)}}f = D_{a^{(1)}} (D_{a^{(2)}} (\ldots D_{a^{(k)}}f))$$ (i.e\ the linear space generated by $a^{(1)}, \ldots, a^{(k)}$ is what is called a ``fast space'' for $f^c$ in~\cite[Definition~5]{SalMan17}).
  \end{itemize}
\end{theorem}
{\em Proof. }
Let $V$ be the linear space defined by the $k$ equations $a^{(1)}\cdot x =\ldots =a^{(k)}\cdot x =0$. In other words, $V$ equals the orthogonal $\langle a^{(1)},\ldots ,a^{(k)}\rangle^\perp$ of the vector space $\langle a^{(1)},\ldots ,a^{(k)}\rangle$ generated by $a^{(1)},\ldots ,a^{(k)}$. Recall that, for an $n\times n$ invertible matrix $M$ over $\F_2$,  we defined in Section \ref{2} the function $\varphi_M:\F_2^n \rightarrow \F_2^n$  as $\varphi_M(x) = xM^T$.
\\
According to Theorem~\ref{carc cod k}, $V$ is a degree-drop space for $f$ if and only if there exists an invertible linear transformation $\varphi_M$ such that, defining $g = f \circ \varphi_M+h$, with $g$ homogeneous of degree $r$ and $\deg(h)<r$, the linear space $W$ defined by the equations $x_1= \ldots =x_k=0$ is a degree-drop space for $g$. We can rewrite these equations as $e_{1}\cdot x= \ldots =e_{k}\cdot x=0$, where $(e_{1},\dots ,e_{n})$ is the canonical basis of $\mathbb F_2^n$ defined in Section~\ref{2}. 
On the other hand, by Lemma~\ref{lem:affine-invariance}, 
  we have $V=\varphi_M(W)$. 
Since, for every $x,y$, we have $\varphi_M(x)\cdot y=x\cdot \varphi_M^*(y)$ where $\varphi_M^*$ is the adjoint operator of $\varphi_M$, whose matrix is the transpose of $M$, we have that $y$ belongs to $V^\perp=\langle a^{(1)},\ldots ,a^{(k)}\rangle$ if and only if $\varphi_M^*(y)$ belongs to $W^\perp=\langle e_{1},\ldots ,e_{k}\rangle$, and we can then assume, without loss of generality thanks to Theorem~\ref{carc cod k}, that $M$ is such that $e_i = a^{(i)}M$ for $i=1, \ldots, k$.
\\
We know that $W$ is a degree-drop space for $g$ if and only if each of the monomials of $g$ contains at least one of the variables $x_1, \ldots, x_k$. Equivalently, each of the monomials of $g^c$ is missing at least one of the variables $x_1, \ldots, x_k$. This is also equivalent to the fact that $\deg(D^{(k)}_{e_1, \ldots, e_k} g^{c}) <\deg(g^c)-k$, since
 for every $g$, we can write $g^c(x_1, \ldots,x_n) = g_1(x_1, \ldots,x_n) + x_1x_2\cdots x_k\ g_2(x_{k+1}, \ldots,x_n)$, where each of the monomials of $g_1$ is missing at least one of the variables $x_1, \ldots, x_k$, and we have $D^{(k)}_{e_1, \ldots, e_k} g^{c} = g_2$, since $D^{(k)}_{e_1, \ldots, e_k}g_1=0$ and $D^{(k)}_{e_1, \ldots, e_k} (x_1x_2\cdots x_k)=1$; the condition $\deg(D^{(k)}_{e_1, \ldots, e_k} g^{c}) <\deg(g^c)-k$ is then equivalent to $g_2=0$ since $g^c$ is homogeneous.
\\
From Proposition \ref{prop:Hou-complement}, for any invertible matrix $M$, we have that $g = f \circ \varphi_{M}+h$ for some $h$ with $\deg(h)<\deg(f)$ if and only if $g^c = f^c \circ \varphi_{(M^T)^{-1}}+h_1$ for some $h_1$ with $\deg(h_1)<\deg(f^c)$. From~\cite[Theorem~7]{SalMan17} we know that
\begin{eqnarray*}
    D^{(k)}_{e_1, \ldots, e_k} g^{c}
    &=& (D^{(k)}_{\varphi_{(M^T)^{-1}}(e_1), \ldots, \varphi_{(M^T)^{-1}}(e_k)} f^{c}) \circ \varphi_{(M^T)^{-1}} +  D^{(k)}_{e_1, \ldots, e_k} h_1\\
    &=& (D^{(k)}_{a^{(1)}, \ldots, a^{(k)}} f^{c})\circ \varphi_{(M^T)^{-1}} +  D^{(k)}_{e_1, \ldots, e_k} h_1.
\end{eqnarray*}
Due to the affine invariance of the degree, we have that
\begin{equation*}
    \deg(D^{(k)}_{a^{(1)}, \ldots, a^{(k)}} f^{c}) = \deg((D^{(k)}_{a^{(1)}, \ldots, a^{(k)}} f^{c})\circ \varphi_{(M^T)^{-1}}) = \deg(D^{(k)}_{e_1, \ldots, e_k} g^{c}+D^{(k)}_{e_1, \ldots, e_k} h_1).
\end{equation*}
Hence, since we have $\deg(D^{(k)}_{e_1, \ldots, e_k} h_1))\le\deg(h_1)-k<\deg(g^c)-k$
we deduce that
$\deg(D^{(k)}_{e_1, \ldots, e_k} g^{c}) <\deg(g^c)-k$ if and only if
$\deg(D^{(k)}_{a^{(1)}, \ldots, a^{(k)}} f^{c})<\deg(g^c)-k$, which concludes the proof.
\hfill $\Box$\\

For the particular case $k=1$, Theorem~\ref{thm:dd-space-fast-space} gives:
\begin{corollary}\label{cor:deg-drop-fast-point}
    Let $f$ be a homogeneous Boolean function of degree $r$ in $n$ variables with $1\le r\le n-1$ and let $a\in\F_2^n\setminus \{\mathbf{0}\}$. The linear hyperplane defined by the equation $a\cdot x = 0$ is a degree-drop hyperplane for $f$ if and only if $a$ is a fast point for $f^c$. Consequently, the number of degree-drop linear hyperplanes of $f$ is equal to the number of fast points of $f^c$.
\end{corollary}

In view of Corollary~\ref{cor:deg-drop-fast-point}, we can use Theorem~\ref{com.scon.ddhyp} to obtain a necessary condition for having no fast points:
\begin{corollary}\label{cor:typeB-fast-points}
 Let $f=\sum_{j=1}^pm_j$ be a homogenous $n$-variable
function of degree $r$ with $2\le r \le n-1$. If for all $i\in \var(f)$  there exists
a term $m_{j_i}$ in $f$ such that: \begin{itemize}
\item $i\in \var(m_{j_i})$
\item for all $t\not\in \var(m_{j_i})$, the monomial $\frac{x_t m_{j_i}}{x_i}$ does not appear in $f$,
\end{itemize}
then $f$ has no fast points.
\end{corollary}
The open problem~\ref{open-problem} is therefore equivalent to:
\begin{open}
  Are all the functions which have no fast points affine equivalent to functions which satisfy the condition of Corollary~\ref{cor:typeB-fast-points}?
\end{open}

We will now move to the problem of counting the number of functions without degree-drop hyperplanes.

Recall that the Gaussian $q$-binomial coefficients are defined for any $q>1$ as
\[
\SqBinom{n}{k}_q = \frac{\prod_{i=n-k+1}^{n} (q^i-1)}{\prod_{i=1}^{k} (q^i-1)}
\]
and the number of vector subspaces of $\F_2^n$ of dimension $k$ is $\SqBinom{n}{k}_2$.


We are now ready to compute the cardinality of $K_{1,r,n}$. 
\begin{theorem}\label{thm:K1-explicit-formula}
 Let $r,n$ be integers such that $1\le r \le n$.
 The number of homogeneous functions of degree $r$ in $n$ variables which do not have any degree-drop hyperplane is
 \begin{equation}\label{eq:K1-explicit-formula}
 |K_{1,r,n} | =  \sum_{i=0}^{r}(-1)^i 2^{\frac{i(i-1)}{2}}\SqBinom{n}{i}_2 \left( 2^{\binom{n-i}{r-i}}-1\right)
\end{equation}
where  $\SqBinom{n}{i}_2$ denotes the 2-Gaussian binomial coefficient.

   For any $j$ with $0\le j\le r$, the number of homogeneous functions of degree $r$ in $n$ variables which have exactly $2^j-1$ degree-drop linear hyperplanes  is equal to
   \begin{equation}\label{eq:explicit-formula-function-dd-hyperpl}
\SqBinom{n}{j}_2 \sum_{i=0}^{r-j}(-1)^i 2^{\frac{i(i-1)}{2}}\SqBinom{n-j}{i}_2 \left( 2^{\binom{n-j-i}{r-j-i}}-1\right).
\end{equation}
In particular, there are no functions with exactly $2^{r-1}-1$ degree-drop hyperplanes.

The number of (not necessarily homogeneous) functions of degree $r$ in $n$ variables which have $2^j-1$ degree-drop linear hyperplanes (respectively no degree-drop hyperplanes) is equal to the quantity in~\eqref{eq:explicit-formula-function-dd-hyperpl} (repectively~\eqref{eq:K1-explicit-formula}) multiplied by $2^{\sum_{\ell=0}^{r-1}\binom{n}{\ell}} $.
\end{theorem}
{\em Proof }
If $r=n$, using the $q$-analogue of the binomial theorem, the sum on the right hand side of \eqref{eq:K1-explicit-formula} equals $0$, and we know from Lemma~\ref{lem:bits-and-bobs}(vi) that indeed $K_{1,n,n} = \emptyset$, so~\eqref{eq:K1-explicit-formula} holds. For the rest of the proof we assume $r<n$.

By Corollary~\ref{cor:deg-drop-fast-point}, we know that the number of elements in  $RM(r,n)/RM(r-1,n)$ which have $2^j-1$ linear degree-drop hyperplanes (in particular, for $j=0$, no hyperplanes) is equal to
the number of elements in $RM(n-r,n)/RM(n-r-1,n)$ which have $2^j-1$ fast points (in particular,  for $j=0$, no fast points). These were computed in~\cite[Theorem~6, Corollary~3]{SalMan17}, and give the results in the statement of the theorem.
\hfill $\Box$

Using Theorem~\ref{thm:K1-explicit-formula} we can examine the probability that a function $f$ of degree $r$ in $n$ variables has a degree-drop hyperplane, which is
\begin{equation}\label{eq:prob}
\frac{2^{\binom{n}{r}}-1 - |K_{1,r,n}|}{2^{\binom{n}{r}}-1} = \sum_{i=1}^{r}(-1)^{i-1} 2^{\frac{i(i-1)}{2}}\SqBinom{n}{i}_2 \frac{ 2^{\binom{n-i}{r-i}}-1}{2^{\binom{n}{r}}-1}.
\end{equation}
For degree $r=2$, this can be computed easily as:
\[
\frac{(2^n-1)(2^{n-1}-1)}{3\left( 2^{\binom{n}{2}}-1\right)}
\]
which has the limit 0 as $n$ tends to infinity.

When $3\le r \le n-3$, as in~\cite[Section~8]{SalMan17}, one can verify that the terms in the sum \eqref{eq:prob} above have alternating signs and decrease rapidly in absolute value. Therefore the probability above can be upper and lower bounded by the first term, respectively first two terms in the sum:
\[
\SqBinom{n}{1}_2\frac{2^{\binom{n-1}{r-1}} -1}{2^{\binom{n}{r}} -1} - 2\SqBinom{n}{2}_2 \frac{2^{\binom{n-2}{r-2}} -1}{2^{\binom{n}{r}} -1}
\le \frac{2^{\binom{n}{r}}-1 - |K_{1,r,n}|}{2^{\binom{n}{r}}-1} \le 
(2^n-1)\frac{2^{\binom{n-1}{r-1}} -1}{2^{\binom{n}{r}} -1}.
\]
The upper bound has the limit 0 as $n$ tends to infinity (regardless of the value of $r$, as long as $3\le r\le n-3$) which can be seen more clearly by approximating it as
\begin{equation}\label{eq:approx-probability}
\frac{(2^n-1)\left(2^{\binom{n-1}{r-1}} -1 \right)}{2^{\binom{n}{r}}-1}\approx \frac{2^{n+\binom{n-1}{r-1}}}{2^{\binom{n}{r}}} = \frac{1}{2^{\binom{n-1}{r}-n}}.
\end{equation}
The probability is already very low for values of $n$ used in practical applications, as illustrated in the following example.

\begin{example}\label{ex1}
  For functions in 7 variables, using
Theorem~\ref{thm:K1-explicit-formula} we obtain that there are $|K_{1,3,7} | = 34355647824$ homogeneous functions of algebraic degree 3 which have no degree-drop hyperplanes.
This gives a probability of $0.00011905$ that a function of degree 3 in 7 variables has a degree-drop hyperplane (the approximation \eqref{eq:approx-probability} of an upper bound for the probability would give 0.00012207).
Using again Theorem~\ref{thm:K1-explicit-formula}, we see that the number of homogeneous functions of degree 3 in 7 variables with 1,3 or 7 degree-drop spaces is
4078732, 0 and 11811 respectively. For algebraic degree 4 we have $|K_{1,4,7} | =
34231364608$,
which gives a probability of $0.00373617$ that a function of degree 4 in 7 variables has a degree-drop hyperplane (the approximation \eqref{eq:approx-probability} of an upper bound for it would give 0.00390625). The number of homogeneous functions of degree 4 in 7 variables with 1,3, 7 or 15 degree-drop spaces is 126046992, 2314956, 0 and 11811 respectively.
\end{example}



\section{Connections to other invariants and parameters}\label{sec:connections-inv} 
We study
a characterisation of the degree-drop spaces $A$ of $f$ in terms of the indicator function of $A$, then we examine a connection with an affine invariant used by~\cite{lang}.

 Let us recall that for any
set $A\subseteq\Bbb{F}_2^n$  the function  $1_A$, called the
indicator of $A$,  is the Boolean function such that $1_A(x)=1$ if
and only if $x\in A$. If $A$ is an affine space of co-dimension $k$ defined by the equations $a_1(x)=0, \ldots, a_k(x)=0$ with $a_i$ affine functions, then $1_A(x)= \prod_{i=1}^{k}(a_i(x) +1)$ and $\deg(1_A)=k$.
We shall first need the next
lemma, which as far as we know, has never been explicitly stated in
a paper, while it is rather basic.
\begin{lemma}\label{lem:indicator}
Let $f$ be an $n$-variable Boolean function. Let $A$ be an affine subspace of $\mathbb{F}_2^n$ and $1_A$ its indicator function.
Then we have
$$\deg (f 1_A)=\deg(f_{\mid A})+\deg(1_A).$$
\end{lemma}
{\em Proof}. Since the algebraic degree is an affine invariant, we
can assume without loss of generality that $A$ is defined by the equations
$x_{n-k+1}=\dots =x_n=1$. Then $1_A(x) = \prod_{i=n-k+1}^nx_i$. The expression of the ANF of $f_{\mid A}$
is obtained from the ANF of $f$ by substituting $x_i$ by 1 for each
$i=n-k+1,\dots, n$. This same expression, when multiplied by $\prod_{i=n-k+1}^nx_i$, gives the ANF of $(f 1_A)(x)$ (since $(f 1_A)(x) =f(x)$ for any $x \in A$ and $(f 1_A)(x) =0$ for any $x\not \in A$).

 Note that $f 1_A$ is  the identically zero function if and only if $f_{\mid A}$ is the identically zero function. Hence, if $f 1_A$ (and therefore $f_{\mid A}$) are identically zero, the equality in the theorem statement follows, recalling that  $\deg(0)=-\infty$ by convention (as defined  in Section~\ref{2}).

Now assume that $f 1_A$ is not the identically zero function.  Note that the ANF of $f_{\mid A}$ and the ANF of $1_A$ have no variables in common, which means that the degree of their product  is the sum of their degrees, i.e.\ $\deg (f 1_A)=\deg(f_{\mid A})+\deg(1_A)$.
This completes the proof.\hfill $\Box$

\begin{proposition}\label{prop:indicator-deg-drop}
For all positive integers $n, r$ and $k$ such that $n\geq k$ and
$1\leq r\leq n-k$, the elements of $K_{k,r,n}$ are those $f$ of
algebraic degree $r$ such that, for every affine space $A$ of
co-dimension $k$, the $n$-variable Boolean function equal to the
product  of $f$ with the indicator $1_A$ of $A$ satisfies:
$$\deg(f1_A)=r+k.$$
\end{proposition}
Indeed, we necessarily have $f1_A\neq 0$ in both cases $f\in K_{k,r,n}$ 
and $\deg (f 1_A)=r+k$, and assuming this condition satisfied, Proposition \ref{prop:indicator-deg-drop} is a direct consequence of Lemma \ref{lem:indicator} and of the fact that $1_A$ has algebraic degree $k$.\\

Hence, the functions with no degree-drop space of co-dimension $k$ are those nonzero $n$-variable Boolean
functions $f$ which, for all spaces $A$ of co-dimension $k$, satisfy $\deg (f 1_A)=\deg(f)+\deg(1_A)$.
Note that a necessary
condition for that is the existence, in the ANF of $f$, of a highest
degree monomial $m$, and in the ANF of $1_A$, of a highest degree
monomial $m'$, such that $m$ and $m'$ are made of disjoint sets of
variables. This condition is not sufficient (take $1_A(x)=x_1+x_2$
and $f(x)=x_1+x_2+1$ for instance).

 In \cite{lang}, one of the affine invariants used in the classification of $RM(r,n)/RM(r-1,n)$ under $\sim_{r-1}$ is $\mathfrak{R}_k$.
 We recall its definition here and then examine its connection to the problem we have been studying.

 The set of homogeneous polynomials of algebraic degree $r$ in $n$ variables over $\F_2$ (more precisely, $RM(r,n)/RM(r-1,n)$,
 with a homogenous polynomial as representative for each equivalence class) can be viewed as a vector space over $\F_2$ of
  dimension $\binom{n}{r}$, with each polynomial identified  with the vector of its coefficients (assuming some order has
  been fixed on the monomials).
  Fixing an element $f\in RM(r,n)/RM(r-1,n)$ and an integer $1\le k \le n-r$, consider
  $\Phi_{f}: RM(k,n)/RM(k-1,n) \rightarrow RM(r+k,n)/RM(r+k-1,n)$ defined as $\Phi_{f}(g)$ being the class of $gf$ in $RM(r+k,n)/RM(r+k-1,n)$. Note that $\Phi_{f}(g) = 0$ if and only if $\deg(gf)<r+k$. It is easy to check that $\Phi_{f}$ is $\F_2$-linear, so $\ker(\Phi_f)$ is a vector space.
  The invariant $\mathfrak{R}_k(f)$ is defined as the dimension of that vector space, i.e.\ $\mathfrak{R}_k(f) = \dim(\ker(\Phi_{f}))$. It can be verified that
  this is indeed an invariant with respect to $\sim_{r-1}$, that is, if $f\sim_{r-1} h$ then $\mathfrak{R}_k(f) = \mathfrak{R}_k(h)$, for all $k = 1, 2, \ldots, n-r$.

 We now examine the connection between the invariant $\mathfrak{R}_k (f)$ and
 the existence of degree-drop spaces of co-dimension $k$ for $f$. We first consider hyperplanes.

 \begin{theorem}\label{thm:R1-connection}
   Let $f$ be a homogeneous Boolean function of algebraic degree $r$.
   Let $V$ be the vector space of maximal dimension such that $H_a = \{x \in \F_2^n : a \cdot x =0\}$ is a degree-drop hyperplane for $f$ for all $a\in V\setminus\{\mathbf{0}\}$. Then $\dim(V) = \mathfrak{R}_1 (f)$.
   In particular $f$ has no degree-drop hyperplanes if and only if $\mathfrak{R}_1 (f)=0$.
 \end{theorem}
  {\em Proof}. Put $d=\dim(V)$ and $t=\mathfrak{R}_1(f)= \dim(\ker(\Phi_{f}))$.
  After a suitable invertible affine transformation, we can assume that $V$ is the space generated by the vectors $e_1, \ldots, e_d$.
 By Corollary~\ref{cor:function-with-given-number-of-hyperplanes}, $f(x_1, \ldots, x_n)=x_1 \ldots x_d h(x_{d+1}, \ldots, x_n)$ with $h$ homogeneous  of
  algebraic degree $r-d$ in $n-d$ variables.
%
  We note that $x_if=f$ for all $i=1, \ldots, d$, so for any of the $2^d$ linear functions $g$ that contain only variables from $\{x_1, \ldots, x_d\}$ we have that $gf$ is
  either null or $gf = f$. In either case,  $\deg(gf)<r+1$ and therefore $g\in \ker(\Phi_{f})$. This means  $t= \dim(\ker(\Phi_{f}))\ge d$.

  For the reverse inequality, again after a suitable invertible affine transformation we can assume that $\ker(\Phi_f)$ is generated by the $t$ elements $x_1, \ldots, x_t$. In other words, $\deg(x_if)<r+1$ for all $i=1, \ldots, t$. Note that $\deg(x_if)<r+1$ if and only if $x_i$ appears in every monomial of $f$. But on the other hand that means that for all $i=1, \ldots, t$
  the hyperplane defined by the equation $x_i=0$ is a degree-drop hyperplane of $f$. Therefore $e_1, \ldots, e_t\in V$, so   $t\le \dim(V) =d$, concluding the proof.
 \hfill $\Box$
\begin{example}\label{ex:8varR1}
 For $n=8$ and $r=4$ we computed the number of polynomials which have degree-drop hyperplanes, i.e.
 $2^{\binom{8}{4}}-1 - |K_{1,4,8} |$ using Theorem~\ref{eq:K1-explicit-formula} and obtained 8761037088127.
 Using the file with the classification of $RM(4,8)/RM(3,8)$ available from~\cite{lang}
 we selected those classes for which $\mathfrak{R}_1 (f)>0$ (note that in the file provided in~\cite{lang}, $n- \mathfrak{R}_1 (f)$ is the first number after the ``$R:$'')
 and concluded that there are 11 classes with that property (8 of them with $\mathfrak{R}_1 (f)=1$, two with $\mathfrak{R}_1 (f)=2$ and one with $\mathfrak{R}_1 (f)=4$).
 Adding the number of elements in each of these classes (denoted by ``Orb'' in the file) we obtained the same value 8761037088127 as above.
 Note that each of these 11 representatives has the property that there is at least one variable which appears in all the monomials;
 in other words, Condition \eqref{c1} is not satisfied, so it is obvious that they have degree-drop hyperplanes.
\end{example}

 Next we consider spaces of higher co-dimension.
  \begin{theorem}\label{thm:R2}
   Let $f$ be a Boolean function of algebraic degree $r$ in $n$ variables  and let  $1\le k \le n-r$. If $f$ has a degree-drop space $A$  of co-dimension $k$, then $\mathfrak{R}_k (f)>0$.
 \end{theorem}
  {\em Proof.} By Proposition~\ref{prop:indicator-deg-drop}, $\deg(f 1_A)<r-k$, so $1_A \in  \ker(\Phi_f)$. Therefore $\mathfrak{R}_k (f) = \dim(\ker(\Phi_f))\ge 1$.
 \hfill $\Box$

 The converse of Theorem~\ref{thm:R2} above is not true. To understand this, note $\ker(\Phi_f)$ could contain non-zero elements (and therefore be of dimension greater than zero) without containing any indicator function of an affine space, and therefore (in light of Proposition~\ref{prop:indicator-deg-drop}) without having any degree-drop space of co-dimension $k$. This situation does indeed occur, as the following example shows:
\begin{example}\label{counterexample}
Consider the following polynomial from the classification
of~\cite{lang}, with 3456 meaning $x_3x_4x_5x_6$;
\[f=3456+2357+1457+1267+1238+1358+1458+2468+1378+3478\]
According to~\cite{lang}, it has $\mathfrak{R}_1 (f)=0$ and
$\mathfrak{R}_2 (f)=2$. However, we checked by computer that $f$
does not have any degree-drop space of co-dimension 2 or less.
\end{example}

\section{Experimental results}\label{sec:experiments}
We have computed the number of degree-drop spaces for all the functions in $n=8$ variables of degrees $3\le r \le n-3$ (the other degrees having been settled for any $n$ in Lemma~\ref{lem:bits-and-bobs} (vii) and Section~\ref{sec:special-cases}).
This section describes and analyzes the results.

In \cite{Hou}, Hou showed that there are 31 non-zero classes of polynomials of degree 3 in 8 variables, under the equivalence $\sim_2$. We recall them in the Appendix.

For each function $f$ among
the representatives  $f_2, \ldots, f_{32}$ of these 31 classes, we considered each linear space $V$ of co-dimension up to 3 and we determined whether $V$ is a degree-drop space for $f$ by computing
the degree of the restriction of $f$ to $V$. We counted
the number of degree-drop linear spaces of $f$ of each co-dimension from 1
to 3. Moreover, for each degree-drop subspace of co-dimension
$k\in \{2,3\}$ of  $f$, we determined whether it is a  ``new'' degree-drop space, in the sense that it is not a  subspace of
 a degree-drop subspace of co-dimension $k-1$.
 These 5 values are presented for each function in Table \ref{table2}, in lexicographically decreasing order. 
 We note that there are only three pairs of polynomials that have the same values, namely ($f_{17}$, $f_{28}$), ($f_{19}$, $f_{30}$) and ($f_{23}$, $f_{32}$).
Based on these calculations, we obtain:
\begin{table*}[ht]
  \centering
   \scriptsize \caption{ Number of degree-drop linear spaces of $f$ of each co-dimension from 1
to 3 for the 31 representatives of degree 3 in 8 variables}\label{table2}
 \begin{tabular}{|l|l|l|l|l|l|}
 \hline
 Representative & co-dim 1  & co-dim 2  & co-dim 2 & co-dim 3  & co-dim 3 \\
 &lin spaces&lin spaces& new lin spaces&lin spaces &new lin spaces\\
 \hline
 $f_2$& 7&  875&    0&  17795&  0
\\$f_3$&    1&  187&    60& 6147&   0
\\$f_7$& 1&   127& 0&   3747&    1080
\\ $f_4$&  0 &  49&  49 & 3059 &   168
\\$f_5$&   0&   35&  35&  2371&    256
\\$f_6$&    0&  21&  21& 1683&    360
\\ $f_8$&   0&  13& 13& 1427&   636
\\$f_9$&    0&  7&  7&  995&    568
\\$f_{13}$& 0&   7&   7&   847& 420
\\$f_{16}$& 0&   7&   7&   739& 312
\\$f_{10}$& 0&   3&  3&   867& 678
\\$f_{29}$& 0&   2&   2&   459& 333
\\$f_{11}$& 0&   1&   1&   563& 500
\\$f_{14}$& 0&   1&   1&   459& 396
\\$f_{15}$& 0&   1&   1&   351& 288
\\$f_{24}$& 0&   1&   1&   307& 244
\\$f_{17}$& 0&  1&  1&  243& 180
\\$f_{28}$& 0&   1&   1&   243& 180
\\$f_{26}$& 0&   1&   1&   135& 72
\\$f_{12}$& 0&   0&  0&  651& 651
\\$f_{31}$& 0&   0&   0&   243& 243
\\$f_{18}$& 0&   0&   0&   167& 167
\\$f_{25}$& 0&  0&  0&  155&    155
\\$f_{19}$& 0&   0&   0&   151& 151
\\$f_{30}$& 0&   0&   0&   151& 151
\\$f_{22}$& 0&   0&   0&   105& 105
\\$f_{23}$& 0&   0&   0&   91&  91
\\$f_{32}$& 0&   0&  0&   91&  91
\\$f_{21}$& 0&  0&  0&   75&  75
\\$f_{20}$& 0&   0&   0&   45&  45
\\$f_{27}$& 0&   0&   0&   15&  15
\\
\hline
\end{tabular}
\end{table*}
\begin{corollary}\label{cor:K12-67vars} We have
 \begin{eqnarray*}
 K_{1,3,7}&=&\{ g: g \sim_2 f \mbox{ for some } f\in \{f_4, f_5, f_6\} \cup \{f_8, f_9, \ldots, f_{12}\}\}\\
K_{1,3,8}&=&\{ g: g \sim_2 f \mbox{ for some } f\in \{f_4, f_5, f_6\} \cup \{f_8, f_9, \ldots, f_{32}\}\}\\
K_{2,3,7}&=&\{g: g \sim_2 f_{12}\}\\
K_{2,3,8}&=&\{ g: g \sim_2 f \mbox{ for some } f\in \{f_{12}, f_{18}, f_{19}, f_{20}, f_{21}, f_{22}, f_{23}, f_{25}, f_{27}, f_{30}, f_{31}, f_{32}\}\\
K_{3,3,7}&=&K_{3,3,8} =\emptyset.
 \end{eqnarray*}
 Therefore $\ds(3,7) = \ds(3,8) = 2$; the functions which are optimal from the point of view of their \rds are the ones in $K_{2,3,7}$ and in $K_{2,3,8}$ for 7 and for 8 variables, respectively.
\end{corollary}
Note that while the probability of a function of degree 3 in 7 variables to have a degree-drop hyperplane is low, namely $0.00011905$ (see Example~\ref{ex1}), the probability of it having a degree-drop space of co-dimension 2 is quite high, 0.605765343 (obtained using the sizes of classes given the Appendix and the $K_{2,3,7}$ given in Corollary~\ref{cor:K12-67vars} above).

 We can also confirm which functions of degree 3 in 8 variables have no degree-drop hyperplanes by using our previous results. 
 The functions $f_2, f_3, f_7$ obviously have degree-drop hyperplanes as they do not satisfy condition~\eqref{c1}. The number of degree-drop hyperplanes is 7,1 and 1 respectively, obtained from Lemma~\ref{lem:bits-and-bobs} (iii) for $f_2$ and from Corollary~\ref{cor:function-with-given-number-of-hyperplanes} and Proposition~\ref{propdsmdegree2} for $f_3$ and $f_7$.  If we add the class sizes (recalled in the Appendix) of the remaining functions in 7 variables, that is, the functions in $K_{1,3,7}$ in Corollary~\ref{cor:K12-67vars},
  we obtain indeed the value of $|K_{1,3,7}|$ computed in 
  in Example~\ref{ex1}.

For functions in 8 variables, out of the functions which appear in Table~\ref{table2} as having no degree-drop hyperplanes, $f_{13}$, $f_{15}$, $f_{16}$, $f_{28}$, $f_{29}$, $f_{30}$ and
$f_{32}$ satisfy Proposition~\ref{prop:card-of-monomial-intersection} for $k=1$, and therefore have no degree-drop hyperplanes. All the remaining functions except
 $f_{22}$  satisfy the conditions of Theorem~\ref{com.scon.ddhyp}. As for  $f_{22}$, using the invertible linear change of
 variable $1\leftarrow  1+3$, $2\leftarrow  2+5+7$ and  $7 \leftarrow 5+7$,  $f_{22}$ becomes a function which
 satisfies  Theorem~\ref{com.scon.ddhyp}. Therefore none of them  have degree-drop hyperplanes.

For degree 5 in 8 variables, the representatives under $\sim_4 $ are the complements $f_{2}^c$,$...,f_{32}^c$ (see Proposition~\ref{prop:Hou-complement}). Since $f_{2}$,$...,f_{12}$ are actually functions in 7 or less variables, when we view them as functions in 8 variables and take the complement we will obtain functions where all the monomials contain the variable $x_8$, and therefore they have a degree-drop hyperplane. 
The other 20 representatives $f_{13}$,$...,f_{32}$ (which are functions in 8 variables but not in less than 8 variables) all satisfy Corollary~\ref{cor:typeB-fast-points} except
$f_{18}$, $f_{21}$, and $f_{22}$. 
 However, these last 3 functions do satisfy Corollary~\ref{cor:typeB-fast-points}  after using the
invertible linear changes of variable $8\leftarrow 1+8$ in $f_{18}$, $3\leftarrow 3+4$
in $f_{21}$ and  $1\leftarrow 1+3$ and $7\leftarrow 5+7$ in $f_{22}$. Therefore all these 20 representatives have no fast points, and by Corollary~\ref{cor:deg-drop-fast-point} their complements have no degree-drop hyperplanes, i.e.\
\[K_{1,5,8}=\{ g: g \sim_4 f \mbox{ for some } f\in \{f_{13}^c, \ldots, f_{32}^c\}\}.\]
From our computer calculations we saw that all the polynomials of degree 5 in 8 variables have degree-drop spaces of co-dimension 2. The number of such degree-drop spaces is given in Table \ref{table3} for those representatives that do not have degree-drop hyperplanes. Therefore $K_{2,5,8}=\emptyset$ and  $\ds(5,8)=1$.
\begin{table}
  \caption{Number of degree-drop spaces of co-dimension 2 for the 20 degree 5 representatives in 8 variables which do not have degree-drop hyperplanes}\label{table3}
  \centering
 \begin{tabular}{|l|l|}
 \hline Representative & codim 2 lin spaces\\
 \hline
$f_{13}^c$& 547\\
$f_{16}^c$& 491\\
$f_{14}^c, f_{29}^c$&   379\\
$f_{15}^c$& 323\\
$f_{17}^c, f_{24}^c, f_{28}^c, f_{31}^c$&   267\\
$f_{18}^c, f_{19}^c, f_{26}^c, f_{30}^c$&   211\\
$f_{22}^c$& 183\\
$f_{21}^c, f_{23}^c, f_{25}^c, f_{27}^c, f_{32}^c$& 155\\
$f_{20}^c$& 127\\
\hline
 \end{tabular}
\end{table}

For degree 4 in 7 variables, the representatives under $\sim_3 $ are the complements $f_{2}^c$,$...,f_{12}^c$. Since $f_{2}$,$...,f_{6}$ are actually functions in 6 or less variables, when we view them as functions in 7 variables and take the complement we will obtain functions where all the monomials contain the variable $x_7$, and therefore they have a degree-drop hyperplane. 
The other 6 representatives $f_{7}$,$...,f_{12}$  all satisfy Corollary~\ref{cor:typeB-fast-points} and therefore   $K_{1,4,7} = \{f_{7}^c$,$...,f_{12}^c\}$ by Corollary~\ref{cor:deg-drop-fast-point}. Adding the sizes of their classes (recalled in the Appendix) we obtain the same value for $|K_{1,4,7}|$ as computed in Example~\ref{ex1}. From our computer calculations we determined that the elements of $K_{1,4,7}$   have, respectively, 315, 147, 91, 91, 35, 91  degree-drop subspaces of co-dimension~2. Therefore $K_{2,4,7}=\emptyset$ and  $\ds(4,7)=1$.

Langevin and Leander
\cite{lang} computed a representative from each of the 998 classes
of functions of degree 4 in 8 variables under the equivalence $\sim_3$.
Again, for each of them we determined by computer calculations the number of linear degree-drop spaces of co-dimension $k=1,2,3$, as well as the number of new degree-drop spaces which are not subspaces of a degree-drop space of lower co-dimension.
The function
$x_1x_2x_3x_4$  has 15 degree-drop hyperplanes, as expected.
The functions $x_1x_2(x_3x_4 + x_5x_6)$ and $x_1x_2(x_3x_4 +
x_5x_6+x_7x_8)$
 each have 3 degree-drop hyperplanes, as expected. A further 8 class representatives have one degree-drop hyperplane.
 They are $x_8f_i$ where $f_i \in \{f_4, f_5, f_6, f_8, \ldots, f_{12}\}$, with $f_i$ being the representatives of function classes of degree 3 in 7 variables under $\sim_2$, as listed in the Appendix.
The remaining representatives of functions have no degree-drop hyperplanes.


The values of $\mathfrak{R}_1(f)$ have been computed in~\cite{lang}
for each class (Note that the values for $\mathfrak{R}_1(f)$ and $\mathfrak{R}_1(f^c)$ are swapped in the original file of~\cite{lang}, see~\cite[Section  6.3.3.1]{Rey18}).
Comparing the values of $\mathfrak{R}_1(f)$ and the
number of degree-drop hyperplanes for each class representative $f$, one can
check that indeed $f$ has $2^{\mathfrak{R}_1(f)}-1$ degree-drop
hyperplanes, as stated in Theorem~\ref{thm:R1-connection}.

 Regarding degree-drop spaces of co-dimension 2, obviously the 11 class representatives that have degree-drop hyperplanes will also have
 degree-drop spaces of co-dimension 2. There were 494 classes having degree-drop spaces of co-dimension 2 but not 1,
 and the remaining 493 classes (that is, about half of all the classes)  did not have any degree-drop spaces of co-dimension 2, so $K_{2,4,8}$ consists of the polynomials in those 493 classes.
 Note that  $\mathfrak{R}_2(f) >0$ for 818 of the polynomials, but only 505 of them have degree-drop spaces of co-dimension 2,
 confirming the fact that the converse of Theorem~\ref{thm:R2} does not hold,
 as illustrated in Example~\ref{counterexample}. All the 998 class representatives have degree-drop spaces of co-dimension 3, i.e.
 $K_{3,4,8} = \emptyset$ and therefore $\ds(4,8)=2$, and the functions in  $K_{2,4,8}$ are optimal from the point of view of their \rds.

Taking into consideration for each class the 3 parameters representing the number of degree-drop spaces of co-dimension 1,2 and 3, there were 137
different triples that appeared.
If we consider instead  the 5 values associated to each
class: the number of degree-drop spaces of co-dimension 1, 2,
3 as well as the number of ``new'' degree-drop spaces of co-dimension
2 and 3,
there are 175 different 5-tuples that appear for the 998 representatives.
This is therefore an invariant worth considering, in conjunction with other invariants, for distinguishing
classes under $\sim_{d-1}$. In \cite{Brier-Langevin}, invariants based on the restrictions of a function to all linear hyperplanes were used for the classification of functions of degree 3 in 9 variables, up to $\sim_2$. The results above show that the restrictions to subspaces of higher co-dimension have good potential too.

\begin{table}
  \centering
  \caption{$\ds(r,n)$ for $n=7,8$}
  \label{table-ds}
\begin{tabular}{c|cccccccc}
\hline
  $n\setminus r$ & 1 & 2 & 3 & 4 & 5 & 6 & 7 & 8 \\
  \hline
  6 & 0 & 2 & 1 & 1 & 0 & 0 & - & - \\
  7 & 0 & 2 & 2 & 1 & 0 & 0 & 0 & - \\
  8 & 0 & 3 & 2 & 2 & 1 & 1 & 0 & 0
\end{tabular}
\end{table}


To conclude this section, we summarise the values of $\ds(r,n)$ for $n=6,7,8$ in Table~\ref{table-ds}. 

\section{Conclusion}
In this paper we commenced the systematic study of functions which maintain their degree when restricted to affine spaces of a given co-dimension $k$. We gave several characterizations of these functions.  The functions that have this property have been fully determined in the class of 
direct sums of monomials and, for $k=1$, the class of symmetric polynomials. An explicit formula is given for the number of functions which maintain their degree when restricted to any hyperplane, but such a formula is still unknown for the case of spaces of co-dimension $k\ge 2$. We explored rich connections  with other existing notions such as fast points and indicator functions. Experimentally, we determined the behaviour of all the functions in 8 variables from the point of view of the stability of their degree under restriction to affine spaces.

However, many aspects remain to be investigated.
We determined for some $r$ and $n$ the value of $\ds(r,n)$, but not in the general case ($\ds(r,n)$ being the maximum $k$ for which there are functions of degree $r$ in $n$ variables which maintain their degree on all spaces of co-dimension $k$).
Constructing families of functions $f$ which achieve this optimum value of degree stability would be of interest for cryptographic applications.  These notions could be further generalized  to vectorial Boolean functions.

\section*{Appendix}
Here we denote the 31 non-zero representatives, of the 31 classes of polynomials of degree 3 in 8 variables in~\cite{Hou},  by $f_2, \ldots, f_{32}$ (with $f_i$ being the same as the function denoted by $F_i$ in~\cite{Hou}) and defined as follows, where 123 means $x_1x_2x_3$. For polynomials in up to 7 variables, we also give in brackets the number of homogeneous polynomials in each equivalence class, as computed  in \cite{lang}.
\begin{eqnarray*}
f_2 & = & 123\, (\mbox{class size } 11811)\\
f_3& =&  123+145\, (\mbox{class size }2314956)\\
f_4& =&  123+456\, (\mbox{class size }45354240)\\
f_5& =&  123+245+346\, (\mbox{class size }59527440)\\
f_6& =&  123+145+246+356+456\, (\mbox{class size }21165312)\\
f_7 & = &127+347+567\, (\mbox{class size }1763776)\\
f_{8}& =& 123+456+147\,(\mbox{class size }2222357760)\\
f_9& = & 123+245+346+147\, (\mbox{class size }238109760)\\
f_{10}& = & 123+456+147+257\, (\mbox{class size } 17778862080)\\
f_{11}& = & 123+145+246+356+456+167\, (\mbox{class size }444471552)\\
f_{12}& = & 123+145+246+356+456+167+247\, (\mbox{class size } 13545799680)\\
f_{13}& = &123+456+178\\ f_{14}& = & 123+456+178+478\\
f_{15}& = &123+245+678+147\\ f_{16}& = &123+245+346+378\\
f_{17}& = &123+145+246+356+456+178\\
f_{18}& = &123+145+246+356+456+167+238\\
f_{19}& = &123+145+246+356+456+158+237+678\\
f_{20}& = &123+145+246+356+456+278+347+168\\f_{21}& = &145+246+356+456+278+347+168+237+147\\ f_{22}& = & 123+234+345+456+567+678+128+238+348+458+568+178\\
f_{23}& = & 123+145+246+356+456+167+578\\
f_{24}& = & 123+145+246+356+456+167+568\\
f_{25}& = & 123+145+246+356+456+167+348\\
f_{26}& = & 123+456+147+257+268+278+348\\
f_{27}& = & 123+456+147+257+168+178+248+358
\\f_{28} & = & 127+347+567+258+368\\
f_{29}& = & 123+456+147+368\\f_{30}& = & 123+456+147+368+578\\
f_{31}& = & 123+456+147+368+478+568\\
f_{32}& = & 123+456+147+168+258+348.
  \end{eqnarray*}
\end{document}